\documentclass[12pt]{amsart}

\usepackage{amsmath}
\usepackage{amssymb}  % defines \nmid, for example
\usepackage{latexsym} % for \Box
\usepackage{comment}
\usepackage{url}
\usepackage{fullpage,url,amssymb,amsmath,amsthm,amsfonts,mathrsfs}
\usepackage[usenames,dvipsnames]{color}
\usepackage[pagebackref = true, colorlinks = true, linkcolor = blue, citecolor = Green]{hyperref}
\usepackage[alphabetic,lite]{amsrefs}
\usepackage{enumitem}
\usepackage{amscd}   % for commutative diagrams
\usepackage[all, cmtip]{xy} % for complicated commutative diagrams
\usepackage{xfrac}

\DeclareFontEncoding{OT2}{}{} % to enable usage of cyrillic fonts
\newcommand{\textcyr}[1]{%
 {\fontencoding{OT2}\fontfamily{wncyr}\fontseries{m}\fontshape{n}\selectfont #1}}

\usepackage[all]{xy}
\usepackage{fullpage}
\newcommand{\Sha}{{\mbox{\textcyr{Sh}}}}
\newcommand{\Bcyr}{{\mbox{\textcyr{B}}}}

\usepackage{color} 
\newcommand{\showappendix}[1]{{#1}}

\newcommand{\defi}[1]{\textsf{#1}} % for defined terms
\def\act#1#2%
  {\mathop{}%
   \mathopen{\vphantom{#2}}^{#1}%
   \kern-3\scriptspace%
   #2}

\newcommand{\Z}{{\mathbb Z}}
\newcommand{\Q}{{\mathbb Q}}
\newcommand{\R}{{\mathbb R}}

\newcommand{\F}{{\mathbb F}}
\newcommand{\A}{{\mathbb A}}
\newcommand{\PP}{{\mathbb P}}
\newcommand{\G}{{\mathbb G}}
\newcommand{\OO}{{\mathcal O}}
\newcommand{\Abar}{{\overline{A}}}

\newcommand{\Kbar}{{\overline{K}}}

\newcommand{\kbar}{{\overline{k}}}

\newcommand{\Xbar}{{\overline{X}}}

\newcommand{\Vbar}{{\overline{V}}}

\newcommand{\calA}{{\mathcal A}}
\newcommand{\calB}{{\mathcal B}}

\newcommand{\calE}{{\mathcal E}}

\newcommand{\calL}{{\mathcal L}}

\newcommand{\To}{\longrightarrow}

\newcommand{\isom}{\simeq}

\DeclareMathOperator{\Num}{Num}

\DeclareMathOperator{\Hom}{Hom}

\DeclareMathOperator{\im}{image}

\DeclareMathOperator{\divv}{div}

\DeclareMathOperator{\coker}{coker}

\DeclareMathOperator{\Gal}{Gal}

\DeclareMathOperator{\Res}{Res}

\DeclareMathOperator{\Br}{Br}
\DeclareMathOperator{\Sym}{Sym}
\DeclareMathOperator{\Div}{Div}

\DeclareMathOperator{\Pic}{Pic}
\DeclareMathOperator{\Alb}{Alb}

\DeclareMathOperator{\HH}{H}
\DeclareMathOperator{\NS}{NS}

\DeclareMathOperator{\Supp}{Supp}
\DeclareMathOperator{\et}{et}

\DeclareMathOperator{\CH}{CH}

\newcommand{\kk}{\mathbf{k}}

\newcommand{\ind}{\operatorname{ind}}
\newcommand{\per}{\operatorname{per}}

\newcommand{\BMd}[1]{\textup{BM}_{#1}}
\newcommand{\BMdd}[1]{\textup{BM}^{\perp}_{#1}}

\newcommand{\BMprimary}[1]{\BMd{#1}}
\newcommand{\BMcoprime}[1]{\BMdd{#1}}

\newtheorem{Theorem}{Theorem}[section]
\newtheorem{Lemma}[Theorem]{Lemma}
\newtheorem{Proposition}[Theorem]{Proposition}
\newtheorem{Corollary}[Theorem]{Corollary}

\newtheorem{Example}[Theorem]{Example}
\newtheorem{Question}[Theorem]{Question}

\numberwithin{equation}{section}

\theoremstyle{remark}
\newtheorem{Remark}[Theorem]{Remark}
\begin{document}
\showappendix

\title{Degree and the Brauer-Manin obstruction}

\author{Brendan Creutz}
\address{School of Mathematics and Statistics, University of Canterbury, Private Bag 4800, Christchurch 8140, New Zealand}
\email{brendan.creutz@canterbury.ac.nz}
\urladdr{http://www.math.canterbury.ac.nz/\~{}bcreutz}

\author{Bianca Viray}
\address{University of Washington, Department of Mathematics, Box 354350, Seattle, WA 98195, USA}
\email{bviray@math.washington.edu}
\urladdr{http://math.washington.edu/\~{}bviray}

\author{an appendix by Alexei N. Skorobogatov}
\address{Department of Mathematics, South Kensington Campus,
Imperial College London, SW7 2BZ England, U.K. -- and --
Institute for the Information Transmission Problems,
Russian Academy of Sciences, 19 Bolshoi Karetnyi, Moscow, 127994
Russia}
\email{a.skorobogatov@imperial.ac.uk}

\thanks{The second author was partially supported by NSA Young Investigator's Award \#H98230-15-1-0054, NSF CAREER grant DMS-1553459, and a University of Canterbury Visiting Erskine Fellowship.}


\begin{abstract}
	Let $X \subset \PP^n_k$ be a smooth projective variety of degree $d$ over a number field $k$ and suppose that $X$ is a counterexample to the Hasse principle explained by the Brauer-Manin obstruction. We consider the question of whether the obstruction is given by the $d$-primary subgroup of the Brauer group, which would have both theoretic and algorithmic implications.
	 We prove that this question has a positive answer in the case of torsors under abelian varieties, Kummer {surfaces} and (conditional on finiteness of Tate-Shafarevich groups) bielliptic surfaces. In the case of Kummer surfaces we show, more specifically, that the obstruction is already given by the $2$-primary torsion, {and indeed that this holds for higher-dimensional Kummer varieties as well.} We construct a conic bundle over an elliptic curve that shows that, in general, the answer is no.
\end{abstract}

\maketitle

%%%%%%%%%%%%%%%%%%%%%%%%%%%%%%%%%%%%%%%%%%%%%%%%%%%%%%%%%%%%%%%%%%%%%%%%%%%%%%%%
\section{Introduction}%%%%%%%%%%%%%%%%%%%%%%%%%%%%%%%%%%%%%%%%%%%%%%%%%%%%%%%%%%
%%%%%%%%%%%%%%%%%%%%%%%%%%%%%%%%%%%%%%%%%%%%%%%%%%%%%%%%%%%%%%%%%%%%%%%%%%%%%%%%
    
    Let $X$ be a smooth projective and geometrically integral variety over a number field $k$.  Manin observed that any adelic point  $(P_v) \in X(\A_k)$ that is approximated by a $k$-rational point must satisfy relations imposed by elements of $\Br X$, the Brauer group of $X$~\cite{Manin1971}.  Indeed, for any element $\alpha \in \Br X:= \HH^2_{\et}(X, \G_m)$, the set of adelic points on $X$ that are orthogonal to $\alpha$, denoted $X(\A_k)^{\alpha}$, {is a closed set containing} the $k$-rational points of $X$.  In particular,
    \[
        X(\A_k)^{\Br} := \bigcap_{\alpha\in \Br X}X(\A_k)^{\alpha} = \emptyset
        \quad \Longrightarrow \quad
        X(k) = \emptyset.
    \]
     
     {   
    In this paper, we investigate whether it is necessary to consider the full Brauer group or whether one can determine \textit{a priori} a proper subgroup $B\subset \Br X$ that \defi{captures the Brauer-Manin obstruction} to the existence of rational points, in the sense that the following implication holds:
    \[
    	X(\A_k)^{\Br} = \emptyset \quad \Longrightarrow \quad X(\A_k)^B := \bigcap_{\alpha \in B} X(\A_k)^\alpha = \emptyset\,.
    \]
    This is of interest from both theoretical and practical perspectives. On the one hand, identifying the subgroups for which this holds may shed considerable light on the nature of the Brauer-Manin obstruction, and, on the other hand, knowledge of such subgroups can facilitate computation Brauer-Manin obstructions in practice.
    
    We pose the following motivating question.
    
	\begin{Question}\label{ques:1}
	 Suppose that $X \hookrightarrow \PP^n$ is embedded as a subvariety of degree $d$ in projective space. Does the $d$-primary subgroup of $\Br X$ capture the Brauer-Manin obstruction to rational points on $X$?
	\end{Question}

	More intrinsically, let us say that \defi{degrees capture the Brauer-Manin obstruction on $X$} if the $d$-primary subgroup of $\Br X$ captures the Brauer-Manin obstruction to rational points on $X$ for all integers $d$ that are the degree of some $k$-rational globally generated ample line bundle on $X$. Since any such line bundle determines a degree $d$ morphism to projective space and conversely, it is clear that the answer to Question \ref{ques:1} is affirmative when degrees capture the Brauer-Manin obstruction.

	\subsection{Summary of results}
		{In general, the answer to Question \ref{ques:1} can be no (see the discussion in \S\ref{sec:discussion}). However, there are many interesting classes of varieties for which the answer is yes.} We prove that degrees capture the Brauer-Manin obstruction for torsors under abelian varieties, for Kummer surfaces, and, assuming finiteness of Tate-Shafarevich groups of elliptic curves, for bielliptic surfaces. {We also deduce (from various results appearing in the literature) that degrees capture the Brauer-Manin obstruction for all geometrically rational minimal  surfaces.}
    
	    Assuming finiteness of Tate-Shafarevich groups, one can deduce the result for torsors under abelian varieties rather easily from a theorem of Manin (see Remark~\ref{Remark:ConditionalProof} and Proposition~\ref{prop:Bcyr}). In \S\ref{sec:AV} we unconditionally prove the following much stronger result.

	\begin{Theorem}
		\label{thm:MainTorsor}
            Let $X$ be a $k$-torsor under an abelian variety, let $B \subset \Br X$ be any subgroup, and let $d$ be any multiple of the period of $X$. In particular, $d$ could be taken to be the degree of a $k$-rational globally generated ample line bundle.
	If $X(\A)^B = \emptyset$, then $X(\A)^{B[d^\infty]} = \emptyset$, where $B[d^\infty] \subset B$ is the $d$-primary subgroup of $B$.
	\end{Theorem}

	This not only shows that degrees capture the Brauer-Manin obstruction (apply the theorem with $B = \Br X$), but also that the Brauer classes with order relatively prime to $d$ cannot provide \emph{any} obstructions to the existence of rational points.

      \begin{Remark}
    As one ranges over all torsors of period $d$ under all abelian varieties over number fields, elements of arbitrarily large order in of $(\Br V)[d^\infty]$ are required to produce the obstruction (see Proposition~\ref{prop:WholePPrimary}).
    \end{Remark}
    
	{We use Theorem \ref{thm:MainTorsor} to deduce that degrees capture the Brauer-Manin obstruction on certain quotients of torsors under abelian varieties. This method is formalized in Theorem~\ref{thm:QuotientOfPHS} and applied in \S\ref{sec:Bielliptic} to prove the following.}

    \begin{Theorem}\label{thm:MainBielliptic}
        Let $X$ be a bielliptic surface and assume that the Albanese torsor $\Alb^1_X$ is not a nontrivial divisible element in the Tate-Shafarevich group, $\Sha(k,\Alb^0_X)$. Then degrees capture the Brauer-Manin obstruction to rational points on $X$.
    \end{Theorem}
    
    \begin{Remark}
        As shown by Skorobogatov~\cite{Skorobogatov}, the Brauer-Manin obstruction is insufficient to explain all failures of the Hasse principle on bielliptic surfaces.  Therefore, if $X$ is a bielliptic surface of degree $d$, then it is possible that there are adelic points orthogonal to the $d$-primary subgroup of $\Br X$ even though $X(k) = \emptyset$.  However, Theorem~\ref{thm:MainBielliptic} shows that in this case one also has that $X(\A_k)^{\Br} \neq \emptyset$.
    \end{Remark}

 Before stating our results on Kummer varieties we fix some notation. We say that $X$ satisfies $\BMprimary{d}$ if the $d$-primary subgroup of $\Br X$ captures the Brauer-Manin obstruction, i.e., if the following implication holds:
        \begin{equation*}
	    	X(\A_k)^{\Br} = \emptyset \quad \Longrightarrow \quad X(\A_k)^{\Br_d} := \bigcap_{\alpha \in (\Br X)[d^\infty]} X(\A_k)^\alpha = \emptyset\,.
    	\end{equation*}
    We say that $X$ satisfies $\BMcoprime{d}$ if there is no prime-to-$d$ Brauer-Manin obstruction, i.e., if the following implication holds 
    \begin{equation*}\label{eq:BMdperp}
        X(\A_k)\neq\emptyset
        \qquad \Longrightarrow \qquad
        X(\A_k)^{\Br_{d^{\perp}}} := \bigcap_{\alpha\in (\Br X)[d^{\perp}]}X(\A_k)^{\alpha} \neq \emptyset\,,
    \end{equation*}
    where $(\Br X)[d^\perp]$ denotes the subgroup of elements of order prime to $d$. {These properties are birational invariants of smooth projective varieties (see Lemma \ref{lem:birationalinvariance}).}
       
    \begin{Remark}
        Note that $\BMprimary{1}$ and $\BMcoprime{1}$ are equivalent; they hold if and only if 
        \[
            X(\A_k)\neq \emptyset \Longleftrightarrow X(\A_k)^{\Br}\neq\emptyset.
        \]
        {More generally, global reciprocity shows that the same is true of $\BMprimary{d}$ and $\BMcoprime{d}$ whenever $(\Br X)[d^\infty] = (\Br k)[d^\infty]$ In general, however, $\BMprimary{d}$ and $\BMcoprime{d}$ are logically independent.}
    \end{Remark}

	{The following theorem is proved in \S\ref{sec:TwistedKummers}. We refer to that section for the definition of a Kummer variety.}

    \begin{Theorem}\label{thm:MainKummer}
        Kummer varieties satisfy $\BMprimary{2}$.
    \end{Theorem}
    {Since the Picard lattice of any K3 surface is even, it is a straightforward consequence of this theorem that degrees capture the Brauer-Manin obstruction on Kummer surfaces.}

    Theorem~\ref{thm:MainKummer} complements the recent result of Skorobogatov and Zarhin that Kummer varieties satisfy $\BMcoprime{2}$~\cite[Theorem 3.3]{SZ16}. {As remarked above, this is logically independent from Theorem~\ref{thm:MainKummer} except when $(\Br X)[2^\infty]$ consists solely of constant algebras. In \cite[Theorem 4.3]{SZ16} it is shown that $(\Br X)[2^\infty] = (\Br_0 X)[2^\infty]$ for Kummer varieties attached to $2$-coverings of products of hyperelliptic Jacobians with large Galois action on $2$-torsion. This is the case of interest in~\cite{HarpazSkorobogatov16} where it is shown that (conditionally on finiteness of Tate-Shafarevich groups) some Kummer varieties of this kind satisfy the Hasse principle. Skorobogatov and Zarhin have shown that these Kummer varieties have no nontrivial $2$-primary Brauer classes \cite[\S\S 4 \& 5]{SZ16}.  Given this, both Theorem~\ref{thm:MainKummer} and \cite[Theorem 3.3]{SZ16} provide an explanation for the absence of any condition on the Brauer group in~\cite[Thms. A and B]{HarpazSkorobogatov16}.

 \showappendix{After seeing a draft of this paper, Skorobogatov noted that it is possible to use our results (specifically Lemma~\ref{lem:killntors}) to extend the proof of \cite[Theorem 3.3]{SZ16} to obtain the following common generalization.

\begin{Theorem}\label{thm:KummerGeneral}
	Let $X$ be a Kummer variety over a number field $k$ and suppose $B \subset \Br X$ is a subgroup. If $X(\A)^B = \emptyset$, then $X(\A)^{B[2^\infty]} = \emptyset$.
\end{Theorem}}

This is the analog {of Theorem~\ref{thm:MainTorsor}} for Kummer varieties.% of Theorem~\ref{thm:MainTorsor} above. 
The proof is given in the appendix by Skorobogatov. 
	%%%%%%%%%%%%%%%%%%%%%%%%%%%%%%%%%%%%%%%%%%%%%%%%%%%%%%%%%%%%%%%%%%%%%%%%%%%%
    \subsection{Discussion}\label{sec:discussion}%%%%%%%%%%%%%%%%%%%%%%
    %%%%%%%%%%%%%%%%%%%%%%%%%%%%%%%%%%%%%%%%%%%%%%%%%%%%%%%%%%%%%%%%%%%%%%%%%%%%

    Question \ref{ques:1} trivially has a positive answer when $X(\A_k)^{\Br} \ne \emptyset$ and, in particular when $X(k) \ne \emptyset$. At the other extreme, the answer is also yes when either $X(\A_k) = \emptyset$ or $\Alb^1_X(\A_k)^{\Br} = \emptyset$ (see Corollary~\ref{cor:albob}). In particular, {the answer is yes for varieties satisfying the Hasse principle.}
	
	Theorem \ref{thm:MainTorsor} shows that the answer is yes for curves of genus $1$. When the first draft of the present paper was made available we were unaware of any example of a higher genus curve $X$ and prime $p$ for which one could show that $X$ does not satisfy $\BMprimary{p}$. Motivated by this, we undertook a deeper study of the Brauer-Manin obstruction on curves, joint with Voloch \cite{CVV}.  There we produced a genus $3$ curve over $\Q$ with a $0$-cycle of degree $1$ that is a counterexample to the Hasse principle explained by the $2$-torsion subgroup of the Brauer group, but with no odd Brauer-Manin obstruction. This {example shows} that degrees do not capture the Brauer-Manin obstrution on higher genus curves, since the $0$-cycle of degree $1$ implies that every sufficiently large integer is the degree of a very ample line bundle.
	
    One reason why degrees cannot capture the Brauer-Manin obstruction in general is that while the set $\{d\in \mathbb{N} \;:\; $X$ \textup{ satisfies } \BMprimary{d}\}$ is a birational invariant, the set of integers that arise as degrees of globally generated ample line bundles on $X$ is not. Exploiting this one can construct examples (even among geometrically rational surfaces) for which degrees do not capture the Brauer-Manin obstruction (See Lemma~\ref{lem:blowup} and Example~\ref{Ex:dp2}). At least in the case of surfaces, this discrepancy can be dealt with by considering only minimal surfaces, {i.e., surfaces which do not contain a Galois-invariant collection of pairwise disjoint $(-1)$-curves}. 
    
    Different and more serious issues are encountered in the case of nonrational surfaces of negative Kodaira dimension. In \S \ref{sec:ConicBundle} we give an example of a minimal conic bundle over an elliptic curve for which degrees do not capture the Brauer-Manin obstruction. This is unsurprising (and somewhat less disappointing) given the known pathologies of the Brauer-Manin obstruction on quadric fibrations~\cite{CPS}. In \S\ref{sec:ConicBundle} we note that degrees do capture the Brauer-Manin obstruction on minimal rational conic bundles and on Severi-Brauer {bundles} over elliptic curves with finite Tate-Shafarevich group (See Theorem~\ref{thm:ConicBundle}).

	For a minimal del Pezzo surface of degree $d$, degrees capture the Brauer-Manin obstruction as soon as $\BMprimary{d}$ holds. This is because the canonical class generates the Picard group, except when the surface is $\PP^2$, a quadric or a rational conic bundle \cite[Theorem 3.9]{Hassett}, in which cases it follows from results mentioned above. Moreover, $\BMprimary{d}$ holds trivially when the degree, $d$, is not equal to $2$ or $3$. Indeed, when $d = 1$ there must be a rational point and when $d  > 3$ the exponent of $\Br X/\Br_0 X$ divides $d$~\cite[Theorem 4.1]{Corn07}\footnote{In fact, it is well-known that del Pezzo surfaces of degree at least $5$ have trivial Brauer group and satisfy the Hasse principle (see, e.g.,~\cite[Thm. 4.1]{Corn07} and~\cite[Thm. 2.1]{VA-dPNotes}).}. Swinnerton-Dyer has shown that any cubic surface such that $\exp(\Br X/\Br_0 X)\nmid 3$ must satisfy the Hasse principle, implying that the answer is yes for $d = 3$~\cite[Cor. 1]{SD-BrauerCubic}. Whether the analogous property holds for del Pezzo surfaces of degree $2$ was considered by Corn~\cite[Question 4.5]{Corn07}, but remained open} until recent work of Nakahara. He showed that odd torsion Brauer classes on a del Pezzo surface of degree $2$ cannot obstruct the Hasse principle~\cite{Masahiro}.

	{Taken together, the results mentioned in the previous two paragraphs combine to yield the following.
	\begin{Theorem}
		Degrees capture the Brauer-Manin obstruction on geometrically rational minimal surfaces over number fields.
	\end{Theorem}
	}
  
    Our results above give an affirmative answer {(conditional on the finiteness of certain Tate-Shafarevich groups in the case of bielliptic surfaces)} for two of the four classes of surfaces of Kodaira dimension $0$, the other two being K3 and Enriques surfaces. {Until quite recently, all known examples of Brauer-Manin obstructions to the existence of rational points on K3 or Enriques surfaces have come from the $2$-torsion subgroup of the Brauer group, implying that degrees capture since the N\'eron-Severi lattice is even {in both cases}. In addition, even among diagonal quartic surfaces over $\Q$ where there were known to be nonconstant elements of odd order, Ieronymou and Skorobogatov showed that $\BMcoprime{2}$ holds \cite{IS,IScorrigendum}. Nakahara has extended this result to some diagonal quartics over general number fields, and, using results of this paper (specifically Lemma~\ref{lem:killntors}) has strengthened this to show that $\BMprimary{2}$ holds and hence that degrees capture \cite{Masahiro}. These diagonal quartics are K3 surfaces that are geometrically Kummer, but not necessarily Kummer over their base field.

    In contrast, some time after the first draft of the present paper was made available, Corn and Nakahara \cite{CornNakahara} produced an example of a degree $2$ K3 surface with a $3$-torsion Brauer-Manin obstruction, showing that $\BMcoprime{2}$ does not hold. This may suggest it is unlikely that degrees capture the Brauer-Manin obstruction for K3 surfaces, though they have not ruled out the possibility of an even transcendental Brauer-Manin obstruction.
 }

 %For these, all known examples of Brauer-Manin obstructions to rational points come from the $2$-torsion subgroup of the Brauer group and the question thus has an affirmative answer, as the N\'eron-Severi lattice is even. However, concrete examples of K3 surfaces exhibiting nonconstant elements of odd order in their Brauer group are scarcely represented in the literature, so this alone should not be taken as serious evidence of a positive answer in general. 
% That being said, we note the work of Ieronymou-Skorobogatov~\cite{IS} showing that there are quite often nonconstant elements of odd order in the Brauer groups of diagonal quartic surfaces over $\Q$ but that these nevertheless satisfy $\BMcoprime{2}$. {This has been extended to some diagonal quartics over general number fields by Nakahara \cite{Masahiro}.} These are K3 surfaces that are geometrically Kummer, but not necessarily Kummer over their base field.

	Taken together, our results and the discussion above indicate that while the degrees of ample line bundles may in general be too crude to determine it, the set of integers $d$ for which $\BMprimary{d}$, or the stronger variant appearing in Theorems~\ref{thm:MainTorsor} and~\ref{thm:KummerGeneral}, holds are interesting birational invariants intimately related to the geometry and arithmetic of the variety.

	%%%%%%%%%%%%%%%%%%%%%%%%%%%%%%%%%%%%%%%%%%%%%%%%%%%%%%%%%%%%%%%%%%%%%%%%%%%%
    \subsection{The analog for $0$-cycles of degree $1$}%%%%%%%%%%%%%%%%%%%%%%
    %%%%%%%%%%%%%%%%%%%%%%%%%%%%%%%%%%%%%%%%%%%%%%%%%%%%%%%%%%%%%%%%%%%%%%%%%%%%

	{One further motivation for studying the question of whether degrees capture the Brauer-Manin obstruction to rational points is that the analogue for $0$-cycles of degree $1$ is expected to hold.} Conjecture $(E)$ on the sufficiency of the Brauer-Manin obstruction for $0$-cycles implies that the $\deg(\calL)$-primary subgroup \emph{does} capture the Brauer-Manin obstruction to $0$-cycles of degree $1$ for any ample globally generated line bundle $\calL$.  See~\S\ref{sec:0cycles} for more details.

%%%%%%%%%%%%%%%%%%%%%%%%%%%%%%%%%%%%%%%%%%%%%%%%%%%%%%%%%%%%%%%%%%%%%%%%%%%%%%%%
\section*{Acknowledgements}%%%%%%%%%%%%%%%%%%%%%%%%%%%%%%%%%%%%%%%%%%%%%%%%%%%%%
%%%%%%%%%%%%%%%%%%%%%%%%%%%%%%%%%%%%%%%%%%%%%%%%%%%%%%%%%%%%%%%%%%%%%%%%%%%%%%%%
    This article was sparked by a discussion at the \textit{Rational Points 2015}; we thank Michael Stoll for organizing a stimulating workshop.  We also thank Jean-Louis Colliot-Th\'el\`ene for comments regarding the origins of Conjecture (E), Evis Ieronymou for pointing out a small error in Section~\ref{sec:0cycles} of an earlier draft, S\'andor Kov\'acs for discussions regarding the proof of Lemma~\ref{lem:BiellipticDivisors}, and Alexei Skorobogatov for the appendix and for helpful comments. We also thank the referee for several {helpful} comments.

%%%%%%%%%%%%%%%%%%%%%%%%%%%%%%%%%%%%%%%%%%%%%%%%%%%%%%%%%%%%%%%%%%%%%%%%%%%%%%%%
\section{{Preliminaries}}\label{sec:Prelim}%%%%%%%%%%%%%%%%%%%%%%%%%%%%%%%%%%%%%%%%%%%%%%%%%%%%%%%
%%%%%%%%%%%%%%%%%%%%%%%%%%%%%%%%%%%%%%%%%%%%%%%%%%%%%%%%%%%%%%%%%%%%%%%%%%%%%%%%

	%%%%%%%%%%%%%%%%%%%%%%%%%%%%%%%%%%%%%%%%%%%%%%%%%%%%%%%%%%%%%%%%%%%%%%%%%%%%
    \subsection{Notation}%%%%%%%%%%%%%%%%%%%%%%%%%%%%%%%%%%%%%%%%%%%%%%%%%%%%%%%
    %%%%%%%%%%%%%%%%%%%%%%%%%%%%%%%%%%%%%%%%%%%%%%%%%%%%%%%%%%%%%%%%%%%%%%%%%%%%

        {For an integer $n$, we define $\Supp(n)$ to be the set of prime numbers dividing $n$. For an abelian group $G$ and integer $d$ we use $G[d]$ to denote the subgroup of elements of order dividing $d$, $G[d^\infty]$ for the subgroup of elements of order dividing a power of $d$ and $G[d^\perp]$ for the subgroup of elements order prime to $d$. }	

        Throughout $k$ denotes a number field, $\Omega_k$ denotes the set of places of $k$, and $\A_k$ denotes the adele ring of $k$.  For any $v\in \Omega_k$, $k_v$ denotes the completion of $k$ at $v$.  We use $K$ to denote an arbitrary field of characteristic $0$. We fix an algebraic closure $\Kbar$ and write $\Gamma_K$ for the absolute Galois group of $K$, with similar notation for $k$ in place of $K$. For {a commutative} algebraic group $G$ {over $k$}, we use $\Sha(k, G)$ to denote the Tate-Shafarevich group of $G$, i.e., the group of torsors under $G$ that are everywhere locally trivial.
        
        Let $X$ be a variety over $K$.  We say that $X$ is \defi{nice} if it is smooth, projective, and geometrically integral. We write $\Xbar$ for the basechange of $X$ to $\Kbar$.  If $X$ is defined over a number field $k$ and $v\in\Omega_k$, we write $X_v$ for the basechange of $X$ to $k_v$.
        
       When $X$ is integral, we use $\kk(X)$ to denote the function field of $X$. The Picard group of $X$, denoted $\Pic X$, is the group of isomorphism classes of $K$-rational line bundles on $X$. {In the case that $X$ is smooth, given a Weil divisor $D \in \Div X$ we denote the corresponding line bundle by $\OO_X(D)$}. The subgroup $\Pic^0 X\subset \Pic X$ consists of those elements that map to the identity component of the Picard scheme.  The N\'eron-Severi group of $X$, denoted $\NS X$, is the quotient $\Pic X/\Pic^0 X$. The Brauer group of $X$, denoted $\Br X$, is $\HH^2_{\et}(X, \G_m)$ and the algebraic Brauer group of $X$ is $\Br_1 X := \ker (\Br X \to \Br \Xbar)$. {The structure morphism yields a map $\Br K \to \Br X$, whose image is the subgroup of constant algebras denoted $\Br_0 X$.}

        \subsection{Polarized varieties, degrees and periods}
        In this paper a \defi{nice polarized variety} over $K$ is a pair $(X,\calL)$ consisting of a nice $K$-variety $X$ and a globally generated ample line bundle $\calL \in \Pic X$. We define the \defi{degree} of a nice polarized variety, denoted by $\deg(X,\calL)$ or $\deg(\calL)$, to be $\dim(X)!$ times the leading coefficient of the Hilbert polynomial, $h(n) := \chi(\calL^{\otimes n})$.
            
        \begin{Lemma}\label{lem:perioddividesdegree}
            Suppose $(X,\calL)$ is a nice polarized variety of degree $d$ over $K$. Then there is a $K$-rational $0$-cycle of degree $d$ on $X$.
        \end{Lemma}
        
        \begin{proof}
            Since $\calL$ is generated by global sections it determines a morphism $\phi_\calL:X \to \PP^N$, for some $N$. Since $\calL$ is ample and globally generated, $\phi_\calL$ is a finite morphism \cite[Corollary 1.2.15, page 28]{Lazarsfeld}. The intersection of $\phi_\calL(X) \subset \PP^N$ with a general linear subvariety of codimension equal to $\dim(X)$ is a $0$-cycle $a$ on $\phi_\calL(X)$ which pulls back to a $0$-cycle of degree $d$ on $X$.
        \end{proof}
            
        For a nice variety $X$ over $K$ and $i\in \Z$ we write $\Alb^i_X$ for degree $i$ component of the Albanese scheme parameterizing $0$-cycles on $X$ up to Albanese equivalence. Then $\Alb^0_X$ is an abelian variety and $\Alb^i_X$ is a $K$-torsor under $\Alb^0_X$. The (Albanese) \defi{period} of $X$, {denoted $\per(X)$,} is the order of $\Alb^1_X$ in the Weil-Ch\^atelet group $\HH^1(K,\Alb^0_X)$. Equivalently, the period is the smallest positive integer $P$ such that $\Alb^P_X$ has a $K$-rational point. Any $k$-rational $0$-cycle of degree $d$ determines a $k$-rational point on $\Alb^d_X$. Thus Lemma \ref{lem:perioddividesdegree} has the following corollary.
        
        \begin{Corollary}\label{cor:perioddividesdegree}
            If $(X,\calL)$ is a nice polarized variety of degree $d$ over $K$, then the period of $X$ divides $d$.
        \end{Corollary}

%%%%%%%%%%%%%%%%%%%%%%%%%%%%%%%%%%%%%%%%%%%%%%%%%%%%%%%%%%%%%%%%%%%%%%%%%%%%%%%%
\subsection{Basic properties of $\BMd{d}$ and $\BMcoprime{d}$}
%%%%%%%%%%%%%%%%%%%%%%%%%%%%%%%%%%%%%%%%%%%%%%%%%%%%%%%%%%%%%%%%%%%%%%%%%%%%%%%%
    
	The definitions of $\BMprimary{d}$ and $\BMcoprime{d}$ yield the following lemma, which we will use freely throughout the paper.
    \begin{Lemma}\label{lem:Support}
    	Let $X$ be a nice variety over a number field $k$, let $d$ and $e$ be positive integers such that $d \mid e$, and set $d_0 := \prod_{p \in \Supp(d)} p$. Then
        \begin{enumerate}
            \item $X$ satisfies $\BMprimary{d}$ (resp. $\BMcoprime{d}$) if and only if $X$ satisfies $\BMprimary{d_0}$ (resp. $\BMcoprime{d_0}$), and 
            \item If $X$ satisfies $\BMprimary{d}$ (resp. $\BMcoprime{d}$), then $X$ satisfies $\BMprimary{e}$ (resp. $\BMcoprime{e}$).
        \end{enumerate}
        In particular, if $d'$ is a positive integer with $\Supp(d) \subset \Supp(d')$ and $X$ satisfies $\BMprimary{d}$ (resp. $\BMcoprime{d}$), then $X$ satisfies $\BMprimary{d'}$ (resp. $\BMcoprime{d'}$).
    \end{Lemma}

	\begin{Lemma}\label{lem:pushpull}	
		Let $\pi\colon Y \to X$ be a morphism of nice varieties over a number field $k$ and let $d$ be a positive integer.
		\begin{enumerate}
		\item\label{case:push} If $Y(\A_k) \ne \emptyset$ and $Y$ satisfies $\BMdd{d}$, then $X$ satisfies $\BMdd{d}$.
		\item\label{case:pull} If $X(\A_k)^{\Br} = \emptyset$ and $X$ satisfies $\BMd{d}$, then $Y$ satisfies $\BMd{d}$.
		\end{enumerate}
	\end{Lemma}
	
	\begin{proof}
	    \begin{enumerate}
			\item Suppose that $X(\A_k)\ne\emptyset$, but $X(\A_k)^{\Br_{d^{\perp}}}= \emptyset$. Then for any $y \in Y(\A_k)$ there exists $\calA \in (\Br X)[d^{\infty}]$ such that $0 \ne (\pi(y),\calA) = (y,\pi^*\calA)$.  Hence $Y(\A_k)^{\Br_{d^\perp}}= \emptyset$ and it follows that $Y$ is not $\BMdd{d}$.
			\item Suppose that $X$ is $\BMd{d}$ and $X(\A_k)^{\Br} = \emptyset$. Then given $y \in Y(\A_k)$, we may find an $\calA \in (\Br X)[d^{\infty}]$ such that $0 \ne (\pi(y),\calA) = (y,\pi^*\calA)$, which shows that $y \notin Y(\A_k)^{\Br_d}$.
		\end{enumerate}	
	\end{proof}
  
	\begin{Lemma}\label{lem:birationalinvariance}
		Let $\sigma\colon Y \dasharrow X$ be a birational map of nice varieties over a number field $k$ and let $d$ be a positive integer. Then
		\begin{enumerate}
			\item\label{it:1} $X$ satisfies $\BMd{d}$ if and only if $Y$ satisfies $\BMd{d}$;
			\item\label{it:2} $X$ satisfies $\BMdd{d}$ if and only if $Y$ satisfies $\BMdd{d}$;
			\item\label{it:3} The map $\sigma$ induces an isomorphism $\sigma^*\colon \Br X \stackrel{\sim}{\to} \Br Y$ such that for any $B \subset \Br X$, $Y(\A_k)^{\sigma^*(B)} = \emptyset$ if and only if $X(\A_k)^{B} = \emptyset$.
		\end{enumerate}
	\end{Lemma}

	\begin{proof}
		By the Lang-Nishimura Theorem~\cites{Nishimura1955, Lang}, $X(\A_k) = \emptyset \Leftrightarrow Y(\A_k) = \emptyset$. With this in mind, the first two statements follow easily from the third. Any birational map between smooth projective varieties over a field of characteristic $0$ can be factored into a sequence of blowups and blowdowns with smooth centers~\cite{WeakFactorizationTheorem}. Hence, it suffices to prove~\eqref{it:3} under the assumption that $\sigma \colon Y \to X$ is a birational morphism obtained by blowing up a smooth center $Z \subset X$. Then $\sigma^*\colon\Br X \to \Br Y$ is an isomorphism~\cite[Corollaire 7.3]{Grothendieck}.
        
		For any field $L/k$ it is clear that $\sigma(Y(L))$ contains $(X\setminus Z)(L)$. Furthermore, since $Z$ is smooth, the exceptional divisor $E_Z$ is a projective bundle over $Z$, so $\sigma(E_Z(L)) = Z(L)$. Therefore $\sigma(Y(L)) = X(L)$. It follows that the map $\sigma\colon Y(\A_k) \to X(\A_k)$ is surjective, and so~\eqref{it:3} follows from functoriality of the Brauer-Manin pairing.
	\end{proof}

		\begin{Lemma}\label{lem:blowup}
			{Let $d$ be a positive integer and let $X$ be} a nice $k$-variety of dimension at least $2$ with the following properties.
			\begin{enumerate}
				\item $X(\A_k) \ne \emptyset$;
				\item $X(\A_k)^{\Br_d}= \emptyset$;
				\item $\Br X/\Br_0 X$ has exponent $d$;
				\item degrees capture the Brauer-Manin obstruction on $X$; and
				\item $X$ has a closed point $P$ with $\deg(P)$ relatively prime to $d$.
			\end{enumerate}
			Let $Y := \textup{Bl}_P X$. 
            Then degrees do not capture the Brauer-Manin obstruction on $Y$.
		\end{Lemma}

		\begin{proof}
			Clearly $X$ does not satisfy $\BMprimary{d'}$ for any $d'$ prime to $d$. By birational invariance, the same is true of $Y$. It thus suffices to exhibit a globally generated ample line bundle on $Y$ of degree prime to $d$. Let $\calL \in \Pic(Y)$ be the pullback of an ample line bundle on $X$ and let $\calE$ denote the line bundle corresponding to the exceptional divisor. For some integer $n$, $\calL^{\otimes n}\otimes \calE^{-1}$ is ample and has degree prime to $d$. An appropriate multiple of this is very ample, hence globally generated, and has degree prime to $d$.
		\end{proof}

		\begin{Example}\label{Ex:dp2}
			An example {of a variety }satisfying the conditions of the lemma with $d = 2$ and {$\deg(P) = 3$} is the del Pezzo surface of degree $2$ given by the equation $w^2 = 34(x^4+y^4+z^4)$ \cite[Remark 4.3]{Corn07}. The blowup of $X$ at a suitable degree $3$ point gives a rational surface for which degrees do not capture the Brauer-Manin obstruction. In particular, the property ``degrees capture the Brauer-Manin obstruction" is not a birational invariant of smooth projective varieties. 
		\end{Example}

%%%%%%%%%%%%%%%%%%%%%%%%%%%%%%%%%%%%%%%%%%%%%%%%%%%%%%%%%%%%%%%%%%%%%%%%%%%%%%%%
\section{The analog for $0$-cycles}\label{sec:0cycles}%%%%%%%%%%%%%%%%%%%%%%%%
%%%%%%%%%%%%%%%%%%%%%%%%%%%%%%%%%%%%%%%%%%%%%%%%%%%%%%%%%%%%%%%%%%%%%%%%%%%%%%%%

	Let $X$ be a nice variety over a number field $k$. The group of $0$-cycles on $X$ is the free abelian group on the closed points of $X$. Two $0$-cycles are directly rationally equivalent if their difference is the divisor of a function $f \in \kk(C)^\times$,	for some nice curve $C \subset X$.   The Chow group of $0$-cycles is denoted $\CH_0X$; it is the group of $0$-cycles modulo rational equivalence, which is the equivalence relation generated by direct rational equivalence. For a place $v\in \Omega$ one defines the modified Chow group
	\[
		\CH_0'X_v = \begin{cases}
            \CH_0X_v & \text{if $v$ is finite;}\\
            \coker\left(N_{\kbar_v/k_v}:\CH_0\Xbar_v \to \CH_0X_v\right) & \text{if $v$ is infinite.}
        \end{cases}
	\]
	Since $X$ is proper there is a well defined degree map $\CH_0X \to \Z$. We denote by $\CH_0^{(i)}X$ the preimage of $i \in \Z$. For an infinite place $v$, the degree of an element in $\CH'_0X_v$ is well defined modulo $2$. 
	
	There is a Brauer-Manin pairing $\prod_v \CH_0X_v \times \Br X \to \Q/\Z$ which, by global reciprocity, induces a complex,
	\[
		\CH_0X \to  \prod_v \CH'_0X_v \to \Hom(\Br X,\Q/\Z)\,.
	\]
	In particular, if there are no classes of degree $1$ in the kernel of the map on the right, then there is no global $0$-cycle of degree $1$.  In this case, we say that there is a Brauer-Manin obstruction to the existence of $0$-cycles of degree $1$.
    
    Conjecture (E) states that the sequence 
	\begin{equation}\label{conjectureE}
		\varprojlim_n (\CH_0X)/n \to  \varprojlim_n \prod_v (\CH_0'X_v)/n \to \Hom(\Br X,\Q/\Z)
	\end{equation}
	is exact. This conjecture has its origins in work of Colliot-Th\'el\`ene and Sansuc~\cite{CTS81}, Kato and Saito~\cite{KS86} and Colliot-Th\'el\`ene~\cite{CT95,CT97}. It has been stated in this form by van Hamel~\cite{vanHamel} and Wittenberg~\cite{Wittenberg-Duke}.

	Conjecture (E) implies that there is a global $0$-cycle of degree $1$ if and only if there is no Brauer-Manin obstruction to such (see, e.g.,~\cite[Rem. 1.1(iii)]{Wittenberg-Duke}). For a nice polarized variety $(X,\calL)$, Conjecture (E) also implies that this obstruction is captured by the $\deg(\calL)$-primary part of the Brauer group. We note that this also follows immediately from \cite[Conjecture 2.2, with $i = \dim X$ and and $l$ a divisor of $\deg(\calL)$]{CT97}.
	
	\begin{Proposition}\label{prop:0cycles}
		Let $(X,\calL)$ be a nice polarized variety of degree $d$ over $k$ and assume that Conjecture (E) holds for $X$. Then there exists a $k$-rational $0$-cycle of degree $1$ on $X$ if and only if there exists $(z_v) \in \prod_v\CH_0^{(1)}X_v$ that is orthogonal to $(\Br X)[d^{\infty}]$.
	\end{Proposition}

	\begin{proof}
		Set $\Q_d := \prod \Q_p$ and $\Z_d := \prod{\Z_p}$, where in both cases the product ranges over the primes dividing $d$. Exactness of \eqref{conjectureE} implies the exactness of its $d$-adic part,
        \[
            \varprojlim_m (\CH_0X)/d^m \to  \varprojlim_m \prod_v (\CH_0'X_v)/d^m \to \Hom((\Br X)[d^\infty],\Q_d/\Z_d)\,.
        \]
		If $(z_v) \in \prod_v\CH_0^{(1)}X_v$ is orthogonal to $(\Br X)[d^{\infty}]$, then by exactness we can find, for every $m \ge 1$, some $z_m \in \CH_0X$ such that for every $v$, $z_m \equiv z_v \pmod{d^m\CH_0'X_v}$. In particular, the degree of $z_m$ is prime to $d$, so there is a global $0$-cycle of degree prime to $d$ on $X$. On the other hand, there is a $0$-cycle of degree $d$ on $X$ by Lemma~\ref{lem:perioddividesdegree}, so there must also be a $0$-cycle of degree $1$ on $X$.
	\end{proof}
	
	{Unconditionally we can show that the $\deg(\calL)$-primary part of the Brauer group captures the Brauer-Manin obstruction to the existence of a global $0$-cycle of degree $1$ when $\Br X/\Br_0 X$ has finite exponent.}
	
    \begin{Proposition}\label{prop:0cyclesfiniteindex}
        Let $(X,\calL)$ be a nice polarized variety of degree $d$. Assume that $\Br X/\Br_0 X$ has finite exponent.  Then there is no Brauer-Manin obstruction to the existence of a $0$-cycle of degree $1$ if and only if there exists $(z_v) \in \prod_v\CH_0^{(1)}X_v$ that is orthogonal to $(\Br X)[d^{\infty}]$. 
    \end{Proposition}
	\begin{proof}
        If there is no Brauer-Manin obstruction to the existence of a $0$-cycles of degree $1$, then by definition there exists a $(z_v) \in \prod_v \CH_0^{(1)}X_v$ that is orthogonal to $\Br X$ and hence to $(\Br X)[d^{\infty}]$. Conversely, suppose there exists $(z_v) \in \prod_v \CH_0^{(1)}X_v$ that is orthogonal to $(\Br X)[d^\infty]$ and let $m\in \Z$ be the maximal divisor of the exponent of $\Br X/\Br_0 X$ that is coprime to $d$.  Since $\deg z_v = \deg z_w$ for all places $v,w$, the adelic $0$-cycle $(z_v)$ is orthogonal to $\Br_0 X$. Hence, the bilinearity of the pairing and the definition of $m$ imply that $(mz_v)$ is orthogonal to $\Br X$.  Furthermore, by Lemma~\ref{lem:perioddividesdegree}, there is a $k$-rational $0$-cycle of degree $d$ on $X$; let $(z) \in \prod_v\CH_0X_v$ be its image. By global reciprocity every integral linear combination of $(z)$ and $(mz_v)$ is orthogonal to $\Br X$. Since $m$ is relatively prime to $d$, some integral linear combination of $(z)$ and $(mz_v)$ has degree $1$, as desired.
    \end{proof}

%%%%%%%%%%%%%%%%%%%%%%%%%%%%%%%%%%%%%%%%%%%%%%%%%%%%%%%%%%%%%%%%%%%%%%%%%%%%%%%%
\section{Torsors under abelian varieties}\label{sec:AV}
%%%%%%%%%%%%%%%%%%%%%%%%%%%%%%%%%%%%%%%%%%%%%%%%%%%%%%%%%%%%%%%%%%%%%%%%%%%%%%%%

	{In this section we prove the following theorem.}
	\begin{Theorem}\label{thm:PreciseMainPHS}
		Let $k$ be a number field, let $Y$ be a smooth projective variety over $k$ that is birational to a $k$-torsor $V$ under an abelian variety, and let $P$ be a positive integer that is divisible by the period of $V$. For any subgroup $B \subset \Br Y$ the following implication holds:
		\[
			Y(\A_k)^B = \emptyset \quad \Longrightarrow \quad  Y(\A_k)^{B[P^\infty]} = \emptyset\,.
		\] 
	\end{Theorem}
    \begin{Remark}
        Theorem~\ref{thm:PreciseMainPHS} is strongest when $P = \per(V)$.  However, determining $\per(V)$ is likely more difficult than determining if $Y(\A_k)^{B[P^\infty]}\neq \emptyset$, so the theorem will often be used for an integer $P$ that is only known to bound the period.
    \end{Remark}

    \begin{proof}[Proof of Theorem~\ref{thm:MainTorsor}]
        This follows immediately from Theorem~\ref{thm:PreciseMainPHS} and Lemma~\ref{lem:perioddividesdegree}.
    \end{proof}

	\begin{Corollary}\label{cor:torsors}
		Every $k$-torsor $V$ of period $P$ under an abelian variety satisfies $\BMd{P}$ and $\BMdd{P}$. 
	\end{Corollary}
    \begin{proof}
        We apply the theorem with $B = \Br X$ and $B = \left(\Br X\right)[P^\perp]$.
    \end{proof}

	\begin{Remark}\label{Remark:ConditionalProof}
		When $\Sha(k,A)$ is finite (i.e., conjecturally always) one can deduce that torsors of period $P$ under $A$ satisfy $\BMd{P}$ using the well known result of Manin relating the Brauer-Manin and Cassels-Tate pairings.  A generalization of Manin's result by Harari and Szamuely~\cite{HarariSzamuely2008} can be used to prove that torsors under semiabelian varieties satisfy $\BMprimary{P}$, conditional on finiteness of Tate-Shafarevich groups (see Proposition~\ref{prop:Bcyr}).
	\end{Remark}

	\begin{Corollary}\label{cor:albob}
		Suppose $X$ is a nice $k$-variety such that $\Alb^1_X(\A_k)^{\Br}= \emptyset$. Then degrees capture the Brauer-Manin obstruction to rational points on $X$.
	\end{Corollary}
	\begin{proof}
		Let $P$ denote the Albanese period of $X$.  By Theorem~\ref{thm:PreciseMainPHS}, there is a $P$-primary Brauer-Manin obstruction to the existence of rational points on $\Alb^1_X$. This pulls back to give a $P$-primary Brauer-Manin obstruction to the existence of rational points on $X$.  We conclude by noting that $P \mid \deg(\calL)$ for any globally generated ample line bundle $\calL \in \Pic X$ by Lemma~\ref{lem:perioddividesdegree}.
	\end{proof}

 {   
 For a nice variety $X$ over $K$ we define the subgroup $\Br_{1/2}X\subset \Br X$ as follows. Let $S$ denote the image of the map $\HH^1(K,\Pic^0\Xbar) \to \HH^1(K,\Pic \Xbar)$ induced by the inclusion $\Pic^0\Xbar \subset \Pic \Xbar$ and let $\Br_{1/2}X$ denote the preimage of $S$ under the map $\Br_1X \to \HH^1(K,\Pic \Xbar)$ that is given by the Hochschild-Serre spectral sequence.}
    
    	\begin{Lemma}\label{lem:killntors}
    		Let $V$ be a $K$-torsor under an abelian variety $A$ over $K$. Let $m$ and $d$ be relatively prime integers with $m$  relatively prime to the period of $V$. If $\Br_{1/2}V$ has finite index in $\Br V$, then there exists an \'etale morphism $\rho\colon V \to V$ such that the induced map,
    		\[
    			\rho^*\colon \frac{\Br V}{\Br_0V} \To \frac{\Br V}{\Br_0V}\,,
    		\]
    		annihilates the $m$-torsion subgroup and is the identity on the $d$-torsion subgroup. {Moreover, one may choose $\rho$ so that it agrees, geometrically, with $[m^r] \colon A \to A$ for some integer $r$.}
    	\end{Lemma}
    	\begin{proof}
    		Since $V$ is a torsor under $A$, there is an isomorphism $\psi \colon\Vbar \to \Abar$ of varieties over $\Kbar$, such that the torsor structure of $V$ is given by $a\cdot v = \psi^{-1}(a + \psi(v))$, for $a \in A(\Kbar)$ and $v \in V(\Kbar)$. Moreover, $\psi$ induces group isomorphisms $\Pic^0\Abar \simeq \Pic^0\Vbar$, $\NS\Abar \simeq \NS\Vbar$, and $\Br\Abar \simeq \Br\Vbar$.
		
    		Let $P$ be the period of $V$ and let $n$ be a power of $m$ such that $n \equiv 1 \bmod Pd$. Then $nV=V$ in $\HH^1(K,A)$, so by \cite[Prop. 3.3.5]{TorsorsAndRationalPoints} $V$ can be made into an $n$-covering of itself. This means that there is an \'etale morphism $\pi\colon V \to V$ such that $\psi\circ\pi = [n] \circ \psi$ where $[n]$ denotes multiplication by $n$ on $A$. {We will show that an iterate of $\pi$ has the desired properties.}

            Since $[n]$ induces multiplication by $n$ on $\Pic^0\Abar$, the morphism $\pi$ induces multiplication by $n$ on $\Pic^0\Vbar$. Indeed, $\pi^* = \psi^*[n]^*(\psi^{-1})^*=\psi^*n(\psi^{-1})^* = n\psi^*(\psi^{-1})^* = n$, where the penultimate equality follows since $\psi^*$ is a homomorphism.  
             Similarly, $\pi$ induces multiplication by $n^2$ on $\NS\Vbar$, since $[n]$ induces multiplication by $n^2$ on $\NS\Abar$. Thus we have a commutative diagram with exact rows
    		\[
    			\xymatrix{
    				\HH^1(K,\Pic^0\Vbar) \ar[r]^i \ar[d]^n& \HH^1(K,\Pic \Vbar) \ar[r]^j\ar[d]^{\pi^*}& \HH^1(k,\NS\Vbar)\ar[d]^{n^2} \\
    				\HH^1(K,\Pic^0\Vbar) \ar[r]^i& \HH^1(K,\Pic \Vbar) \ar[r]^j& \HH^1(k,\NS\Vbar)
    			}
    		\]
    		We claim that $(\pi^2)^*$ annihilates the $n$-torsion of $\HH^1(K,\Pic \Vbar)$. Since $\Br_1V/\Br_0V$ embeds into $\HH^1(K,\Pic \Vbar)$, this would imply that $(\pi^2)^*$ annihilates the $n$-torsion in $\Br_1V/\Br_0V$.  Let us prove the claim. For any $x \in \HH^1(K,\Pic \Vbar)[n]$, we have that $j(\pi^*(x)) = n^2 j(x) = j(n^2x) = 0$, so $\pi^*(x) = i(y)$ for some $y \in \HH^1(K,\Pic^0\Vbar)$ such that $ny \in \ker(i)$. Then $(\pi^2)^*(x) = \pi^*(i(y))  = i(ny) = 0$, as desired. 
		
    		Multiplication by $n$ on $A$ induces multiplication by $n^2$ on $\Br\Abar$ (see~\cite[Middle of page 182]{Berkovich}). Thus $\pi^*$ acts as multiplication by $n^2$ on $\Br\Vbar$, and we have a commutative diagram with exact rows,
    		\[
    			\xymatrix{
    				0 \ar[r]&\Br_1V/\Br_0V \ar[r]^{i'} \ar[d]^{\pi^*}& \Br V/\Br_0V \ar[r]^{j'}\ar[d]^{\pi^*}& \Br V/\Br_1V \ar[d]^{n^2} \\
    				0\ar[r]&\Br_1V/\Br_0V  \ar[r]^{i'}&\Br V/\Br_0V \ar[r]^{j'}&\Br V/\Br_1V
    			}
    		\]
    		The $n$-torsion in $\Br_1V/\Br_0V$ is killed by $(\pi^2)^*$ and  $i'$ is injective, so a similar diagram chase as above shows that $(\pi^3)^*$ kills the $n$-torsion, and hence the $m$-torsion, in $\Br V/\Br_0V$.
		
    		It thus suffices to show that some power of $(\pi^3)^*$ is the identity map on the $d$-torsion subgroup of $\Br V /\Br_0V$. {By definition, the image of the composition $(\Br_{1/2}V \to \Br_{1} V\to \HH^1(K,\Pic \Vbar))$ is contained in $\HH^1(K,\Pic^0 \Vbar)$}, so $\pi^*$ acts as multiplication by $n$ on $\Br_{1/2}V/\Br_0V$. In particular, since $n \equiv 1 \bmod d$, $\pi^*$ acts as the identity on the $d$-torsion subgroup of $\Br_{1/2}V/\Br_0V$. Additionally, since the degree of $\pi^3$ is relatively prime to $d$, the induced map $(\pi^3)^*:(\Br V)[d^\infty] \to (\Br V)[d^\infty]$ is injective (see \cite[Prop. 1.1]{ISZ}). Together this shows that $(\pi^3)^*$ is injective on $(\Br V/\Br_{1/2} V)[d]$. 
            Since $(\Br V/\Br_{1/2}V)[d]$ is finite, some power $\sigma = \pi^{3s}$ of $\pi^3$ acts as the identity on it.
            Thus, for every $\calA \in \Br V$ such that $d\calA\in \Br_0 V$ there is some $\calA' \in \Br_{1/2}V$ such that 
            \[
                \sigma^*(\calA) = \calA + \calA' \quad\textup{and}\quad
                d\calA'\in \Br_0 V.
            \]
            Again $\sigma^*$ is the identity on $(\Br_{1/2}V/\Br_0V)[d]$, so by induction we get $(\sigma^d)^*(\calA) \equiv \calA + d\calA' \equiv \calA \pmod{\Br_0V}$. Therefore $\rho = \sigma^d$ has the desired properties.
    	\end{proof}

    	\begin{Lemma}\label{lem:Br1/2}
		Let $k$ be a number field. If $V$ is a $k$-torsor under an abelian variety, then $\Br_{1/2}V$ has finite index in $\Br V$. 
    	\end{Lemma}
    	\begin{proof}
    		We have a filtration
    		\[
    			\Br_{1/2}V \subset \Br_1V \subset \Br V\,.
    		\]
            The second inclusion has finite index by~\cite[Theorem 1.1]{SZ-Finiteness}. 
            Now we consider the first inclusion.  By the definition of $\Br_{1/2} V$, the quotient $\Br_1 V/\Br_{1/2} V$ injects into the cokernel of the map $\HH^1(k, \Pic^0 \Vbar)\to\HH^1(k, \Pic \Vbar)$.  Then by the long exact sequence in cohomology associated to the short exact sequence
            \[
                0 \to \Pic^0 \Vbar \to \Pic \Vbar \to \NS \Vbar \to 0,
            \]
            the cokernel in question injects into $\HH^1(k, \NS \Vbar)$.  The result now follows since $\NS \Vbar$ is finitely generated and torsion-free.
    	\end{proof}

	\begin{Lemma}\label{lem:compactness}
		Let $X$ be a smooth proper variety over a number field $k$ and let $B\subset \Br X$ be a subgroup.  If $X(\A_k)^B = \emptyset$, then there is a finite subgroup $\tilde{B} \subset B$ such that $X(\A_k)^{\tilde{B}}= \emptyset$.
	\end{Lemma}
	
	\begin{proof}
		By hypothesis $X(\A_k)$ is compact. The lemma follows from the observation that 
		\[
			X(\A_k)^B = \bigcap_{\calA \in B}X(\A_k)^\calA\,.
		\]
		is an intersection of closed subsets of $X(\A_k)$. If the intersection of these subsets is empty, then there is some finite collection of subsets whose intersection is empty. Since $\Br X$ is torsion and a finitely generated torsion abelian group is finite, this completes the proof.
	\end{proof}
    
    	\begin{proof}[Proof of Theorem~\ref{thm:PreciseMainPHS}]
    		In light of Lemma~\ref{lem:birationalinvariance}, we may assume that $Y = V$. 
		
		If $V(\A_k)^{B} = \emptyset$, then, by Lemma~\ref{lem:compactness}, there is a finite subgroup $B' \subset B$ with $V(\A_k)^{B'} = \emptyset$.  Since $V(\A_k)^{B[P^\infty]} \subset V(\A_k)^{B'[P^\infty]}$, the desired implication holds if
    		\[
    			V(\A_k)^{B'} = \emptyset \quad\Longrightarrow\quad V(\A_k)^{B'[P^\infty]} = \emptyset\,,
    		\]
            for all finite subgroups $B' \subset B$. Thus, it suffices to prove the theorem when $B$ is finite.

            	Let $d$ be the exponent of $B[P^\infty]$ and let $m := \textup{exponent}(B)/d$ so that $m$ and $d$ are relatively prime.  Since we are working over a number field, Lemma~\ref{lem:Br1/2} allows us to apply Lemma~\ref{lem:killntors}. Let $\rho\colon V \to V$ be the morphism given by Lemma~\ref{lem:killntors}; then the functoriality of the Brauer-Manin pairing and global reciprocity give that 
            \[
                V(\A_k)^{B} \supset \rho(V(\A_k)^{\rho^*(B)}) = \rho(V(\A_k)^{\rho^*(B[P^\infty])}) = \rho(V(\A_k)^{B[P^\infty]}) \,.
            \]
    		In particular, if $V(\A_k)^B$ is empty, then so must be $V(\A_k)^{B[P^\infty]}$.
    	\end{proof}
	
    %%%%%%%%%%%%%%%%%%%%%%%%%%%%%%%%%%%%%%%%%%%%%%%%%%%%%%%%%%%%%%%%%%%%%%%%%%%%
    \subsection{A conditional extension to semiabelian varieties}%%%%%%%%%%%%%%%
    %%%%%%%%%%%%%%%%%%%%%%%%%%%%%%%%%%%%%%%%%%%%%%%%%%%%%%%%%%%%%%%%%%%%%%%%%%%%
    
     	\begin{Proposition}\label{prop:Bcyr}
    		Let $k$ be a number field and let $V$ be a $k$-torsor under a semiabelian variety $G$ with abelian quotient $A$.  Assume that $\Sha(k,A)$ is finite. Then $V$ satisfies $\BMd{d}$ for any integer $d$ that is a multiple of the period of $V$.
    	\end{Proposition}
    
        \begin{proof}
    		It is known that the obstruction coming from the group
        	\[
        		\Bcyr(V) := \ker\left(\Br_1V \to \bigoplus_v \Br_1V_v/\Br_0V_v\right)
        	\]
        	is the only obstruction to rational points on $V$~\cite[Theorem 1.1]{HarariSzamuely2008}. The extreme cases where $G$ is an abelian variety or a torus are due to Manin~\cite[Th\'eor\`eme 6]{Manin1971} and Sansuc~\cite [Corollaire 8.7]{Sansuc1981}, respectively. The proof of this fact is as follows. There is a homomorphism $\iota\colon\Sha(k,G^*) \to \Bcyr(X)$ (where $G^*$ denotes the $1$-motive dual to $G$) and a bilinear pairing of torsion abelian groups $\langle\,,\,\rangle_{\textup{CT}}:\Sha(k,G) \times \Sha(k,G^*) \to \Q/\Z$. This is related to the Brauer-Manin pairing via $\iota$ in the sense that for any $\beta \in \Sha(k,G^*)$ and $(P_v) \in V(\A_k)$, one has $\langle [V], \beta \rangle_{\textup{CT}} = ( (P_v),\iota(\beta) )$. The assumption that $\Sha(k,A)$ is finite implies that the pairing $\langle\,,\,\rangle_{\textup{CT}}$ is nondegenerate. In particular, if $V(k) \subset V(\A_k)^{\Br} = \emptyset$, then there is some $\calA \in \Bcyr$ such that $V(\A_k)^\calA = \emptyset$. It follows from bilinearity that, if there is such an obstruction, it will already come from the $\per(V)^{\infty}$-torsion elements in $\Bcyr(X)$.
    	\end{proof}

    %%%%%%%%%%%%%%%%%%%%%%%%%%%%%%%%%%%%%%%%%%%%%%%%%%%%%%%%%%%%%%%%%%%%%%%%%%%%
    \subsection{Unboundedness of the exponent}%%%%%%%%%%%%%%%%%%%%%%%%%%%%%%%%%
    %%%%%%%%%%%%%%%%%%%%%%%%%%%%%%%%%%%%%%%%%%%%%%%%%%%%%%%%%%%%%%%%%%%%%%%%%%%%
        One might ask if we can restrict consideration to even smaller subgroups of $\Br V$, for instance if there is an integer $d$ that can be determined \textit{a priori} such that {the $d$-torsion (rather than the $d$-primary torsion) captures the Brauer-Manin obstruction.}
        The following proposition shows that this is not possible for torsors under abelian varieties, at least if Tate-Shafarevich groups of elliptic curves are finite.
	    
    	\begin{Proposition}\label{prop:WholePPrimary}
    		Suppose that Tate-Shafarevich groups of elliptic curves over number fields are finite. For any integers $P$ and $n$ there exists, over a number field $k$, a torsor $V$  under an elliptic curve with $\per(V) = P$,  $V(\A_k)^{(\Br V)[P^n]} \ne \emptyset$, and $V(\A_k)^{(\Br V)[P^{n+1}]} = \emptyset$.
    	\end{Proposition}
        The following lemma will be helpful in the proof.
    	\begin{Lemma}\label{lem:ArbitraryOrderSha}
    		For any integer $N$ there exists an elliptic curve $E$ over a number field $k$ such that $\Sha(k,E)$ has an element of order $N$.
    	\end{Lemma}
    	\begin{Remark}
    		It would be very interesting to know if the curve $E$ can be taken to be defined over $\Q$. To the best of our knowledge, this is unknown for $P$ a sufficiently large prime and for $P$ an arbitrary power of any single prime. It follows from the lemma that this does hold for abelian varieties over $\Q$, since restriction of scalars gives an element of order $N$ in $\Sha(\Q,\Res_{k/\Q}(E))$.
    	\end{Remark}
    	\begin{proof}
		{Recall that the index of a variety $X$ over a field is the gcd of the degrees of the closed points on $X$.} 
    		By work of Clark and Sharif \cite{ClarkSharif2010} we can find a torsor $V$ under an elliptic curve $E/k$ of period $N$ and index $N^2$.  Moreover, the proof of loc. cit. shows that we can find such $V$ with $V(k_v) = \emptyset$ for exactly two primes $v$ of $k$, both of which are finite, prime to $N$ and such that $E$ has good reduction. Let $L/k$ be any degree $N$ extension of $k$ which is totally ramified at both of these primes. By a result of Lang and Tate \cite{LangTate1958} we have that $V(L_w) \ne \emptyset$ for $w$ a place lying over either of these totally ramified primes, and hence $V_L \in \Sha(L,E_L)$. On the other hand, the index of $V_L$ can drop at most by a factor of $N = [L:k]$. Since $V_L$ is locally trivial its period and index are equal \cite[Theorem 1.3]{CasselsIV}. Therefore we have $N \le \textup{Index}(V_L) = \textup{Period}(V_L) \le \textup{Period}(V) = N$, so $V_L$ has order $N$ in $\Sha(L,E_L)$. 
    	\end{proof}

    	\begin{proof}[Proof of Proposition~\ref{prop:WholePPrimary}]
    		By the lemma above, we can find an elliptic curve $E$ such that $\Sha(k,E)$ contains an element of order $P^{n+1}$. Since $\Sha(k,E)$ is finite, we can find $W \in \Sha(k,E)$ such that $V := P^nW$ is not divisible by $P^{n+1}$ in $\Sha(k,E)$. Then a theorem of Manin~\cite{Manin1971} shows that $V(\A_k)^{(\Br V)[P^{n+1}]} = \emptyset$. On the other hand, $\pi \colon W \to V$ is a $P^n$-covering and, by descent theory, the adelic points in $\pi(W(\A_k))$ are orthogonal to $(\Br_1V)[P^n] = (\Br V)[P^n]$ (here we have used the fact that $V$ is a curve and applied Tsen's Theorem).
    	\end{proof}

%%%%%%%%%%%%%%%%%%%%%%%%%%%%%%%%%%%%%%%%%%%%%%%%%%%%%%%%%%%%%%%%%%%%%%%%%%%%%%%%
\section{Quotients of torsors under abelian varieties}\label{sec:QAV}
%%%%%%%%%%%%%%%%%%%%%%%%%%%%%%%%%%%%%%%%%%%%%%%%%%%%%%%%%%%%%%%%%%%%%%%%%%%%%%%%

	\begin{Theorem}\label{thm:QuotientOfPHS}
        Let $k$ be a number field.  Let $Y$ be a smooth projective variety and let $d$ be a positive integer such that for any subgroup $B\subset \Br Y$ we have
        \[ 
            Y(\A_k)^B = \emptyset \Rightarrow Y(\A_k)^{B[d^\infty]} = \emptyset.
        \]
        %Let $Y$ be a smooth projective variety that is $k$-birational to a $k$-torsor $V$ under an abelian variety. 
        %Let $d$ be a positive integer such that $\Supp(\textup{per}(V)) \subset \Supp(d)$ and 
        Let $\pi \colon Y \to X$ be a finite flat cover such that one of the following holds:
        \begin{enumerate}
            \item\label{case:1} $d = 2$, $\pi$ is a ramified double cover, and $(\Br X/\Br_0 X)[2^\infty]$ is finite, or
            \item\label{case:2} $\pi$ is a torsor under an abelian $k$-group of exponent dividing $d$.
        \end{enumerate}
        Then $X$ satisfies $\BMprimary{d}$.  
	\end{Theorem}

    \begin{Remark}
        In fact, a stronger result holds.  There is a subgroup $B\subset (\Br X)[d]$ such that if $X(\A_k)^B\ne \emptyset$, then $X$ is also $\BMcoprime{d}$.  If we are in Case~\eqref{case:1}, then this subgroup $B$ is finite, depends only on the Galois action on the geometric irreducible components of the branch locus of $\pi$, and is generically trivial.
        
        The idea is that the subgroup $B$ controls whether there is a twist of $Y$ over $X$ that is everywhere locally soluble.  If there is such a twist, then we may apply Lemma~\ref{lem:pushpull}\eqref{case:push} and Theorem~\ref{thm:PreciseMainPHS} {to conclude that $X$ inherits $\BMcoprime{d}$ from the covering.}
    \end{Remark}
    \begin{Remark}\label{rmk:PHSSatisfyAssumption}
        Theorem~\ref{thm:PreciseMainPHS} implies that the first hypothesis of Theorem~\ref{thm:QuotientOfPHS} is satisfied when $Y$ is birational to a torsor $V$ under an abelian variety and $d$ is an integer that is divisible by every prime in $\Supp(\per(V))$. 
    \end{Remark}
    
    For the proof, we need a slight generalization of a result of Skorobogatov and Swinnerton-Dyer~\cite[Theorem 3 and Lemma 6]{SSD-2Descent}.
    \begin{Proposition}\label{prop:SSD}
        Let $\pi\colon Y \to X$ be a ramified double cover over $k$ and assume that $\left(\Br X/\Br_0 X\right)[2^{\infty}]$ is finite.  Then
        \[
            X(\A_k)^{\Br_2} \ne\emptyset 
            \quad\Longleftrightarrow \quad
            \bigcup_{a\in k^{\times}/k^{\times2}} \pi^{a}\left(Y^{a}(\A_k)^{(\pi^{a})^*(\Br X[2^{\infty}])}\right)\ne\emptyset,
        \]
        where $\pi^a \colon Y^a \to X$ denotes quadratic twist of $Y\to X$ by $a$.
    \end{Proposition}
    \begin{proof}
        The backwards direction follows from the functoriality of the Brauer group.  Thus we consider the forwards direction.  This proof follows ideas from~\cite[Section 5]{SSD-2Descent}.  We repeat the details here for the reader's convenience.
        
        Let $f\in \kk(X)^{\times}$ be such that $\kk(Y) = \kk(X)(\sqrt{f})$.  We define a finite dimensional $k$-algebra 
        \[
            L := \bigoplus_{\substack{D\in X^{(1)}\\ v_D(f) \textup{ odd}}}\left(\kbar \cap \kk(D)\right).
        \]
        Note that $L$ is independent of the choice of $f$, since the class of $f$ in $\kk(X)^{\times}/\kk(X)^{\times2}$ is unique.  Let $\alpha_1, \dots, \alpha_n \in (\Br X)[2^{\infty}]$ be representatives for the finitely many classes in $(\Br X/\Br_0 X)[2^\infty]$.  Let $S$ be a finite set of places such that for all $v\notin S$, for all $P_v\in X(k_v)$, and for all $1\leq i \leq n$, $\alpha_i(P_v) = 0\in \Br k_v$. (It is well known that finding such a finite set $S$ is possible, see, e.g.,~\cite[Section 5.2]{TorsorsAndRationalPoints}.)  After possibly enlarging $S$ we may assume that $S$ contains all archimedean places and all places that are ramified in a subfield of $L$. We may also assume that $Y^a(k_v)\neq \emptyset$ for all $v\notin S$ and all $a\in k^{\times}/k^{\times2}$ with $v(a)$ even~\cite[Prop. 5.3.2]{TorsorsAndRationalPoints}.
        
        Let $(P_v)\in X(\A_k)^{\Br_2}$.  For $v\in S$, let $Q_v\in X(k_v)$ be such that $a_v:=f(Q_v)\in k_v^{\times}$ and be sufficiently close to $P_v$ so that $\alpha_i(P_v) = \alpha_i(Q_v)$ for all $i$.  For $v\notin S$, set $a_v :=1 $.  Let $c\in k^{\times}$ be such that the class of $c$ lies in the kernel of the natural map $k^\times/k^{\times 2}\to L^\times/L^{\times 2}$. Then by~\cite[Theorem 3]{SSD-2Descent}, the quaternion algebra $\calA = (c,f)_2$ lies in $(\Br X)[2]$.  Using the aforementioned properties of $S$ and the definition of $P_v$ and $a_v$, we then conclude
		\[
			\sum_{v\in \Omega_k} \textup{inv}_v((c, a_v)) =
			\sum_{v\in S}        \textup{inv}_v((c, a_v)) =
			\sum_{v\in S}         \textup{inv}_v(\calA(P_v)) =
			\sum_{v\in\Omega_k} \textup{inv}_v(\calA(P_v)) = 0.
		\]
        Hence, by~\cite[Lemma 6(ii)]{SSD-2Descent}, there exists an $a\in k^{\times}$ with $a/a_v\in k_v^{\times 2}$ such that $Y^a(\A_k) \neq \emptyset.$  Since $a/a_v\in k_v^{\times 2}$ for all $v\in S$ we may further assume that there exists an $(R_v) \in Y^a(\A_k)$ such that $\pi^a(R_v) = Q_v$ for all $v\in S$.  We then have
        \begin{align*}
            \sum_{v}\textup{inv}_v(((\pi^a)^*(\alpha_i))(R_v)) & = 
			\sum_v\textup{inv}_v(\alpha_i(\pi^a(R_v)))
            =  \sum_{v\in S}\textup{inv}_v(\alpha_i(Q_v)) = 
			\sum_{v\in S}\textup{inv}_v(\alpha_i(P_v))\\&
            = \sum_{v}\textup{inv}_v(\alpha_i(P_v))=0,
        \end{align*}
        so $(R_v) \in Y^a(\A_k)^{(\pi^a)^*(\Br X[2^\infty])}$.
    \end{proof}

	\begin{proof}[Proof of Theorem~\ref{thm:QuotientOfPHS}]
        If $\pi$ is ramified, let $G = \Z/2$; otherwise let $G$ be the finite $k$-group such that $\pi$ is a torsor under $G$.
		    For each $\tau\in \HH^1(k, G)$, let $\pi^{\tau}\colon Y^{\tau}\to X$ denote the twisted cover, and define $B^{\tau} := ({\pi^{\tau}})^*(\Br X)$.
        
        Assume that $X(\A_k)^{\Br_d} \neq \emptyset$.  By~\cite[Prop. 1.1]{ISZ}, ${(\pi^{\tau})}^*((\Br X)[d^{\infty}]) = B^{\tau}[d^\infty]$. Therefore, Proposition~\ref{prop:SSD} and descent theory show in cases~\eqref{case:1} and~\eqref{case:2} respectively, that there exists a $\tau\in \HH^1(k, G)$ such that $Y^{\tau}(\A_k)^{B^{\tau}[d^\infty]} \neq\emptyset$. Applying the assumption %Theorem~\ref{thm:PreciseMainPHS}
         with $B=B^\tau$ we conclude that $Y^{\tau}(\A_k)^{B^{\tau}}$ is nonempty. By the functoriality of the Brauer-Manin pairing, we have that $X(\A_k)^{\Br} \neq \emptyset.$
	\end{proof}

    %%%%%%%%%%%%%%%%%%%%%%%%%%%%%%%%%%%%%%%%%%%%%%%%%%%%%%%%%%%%%%%%%%%%%%%%%%%%
    \subsection{Bielliptic Surfaces}\label{sec:Bielliptic}%%%%%%%%%%%%%%%%%%%%%%
    %%%%%%%%%%%%%%%%%%%%%%%%%%%%%%%%%%%%%%%%%%%%%%%%%%%%%%%%%%%%%%%%%%%%%%%%%%%%

        We say that a nice variety $X$ over $K$ is a \defi{bielliptic surface} if $X$ is a minimal algebraic surface of Kodaira dimension $0$ and irregularity $1$. 
    
        %%%%%%%%%%%%%%%%%%%%%%%%%%%%%%%%%%%%%%%%%%%%%%%%%%%%%%%%%%%%%%%%%%%%%%%%
        \subsubsection{Geometry of bielliptic surfaces}%%%%%%%%%%%%%%%%%%%%%%%%%
        %%%%%%%%%%%%%%%%%%%%%%%%%%%%%%%%%%%%%%%%%%%%%%%%%%%%%%%%%%%%%%%%%%%%%%%%
    
            If $X$ is a bielliptic surface over $K$ then it is well known that $\Xbar$ is isomorphic to $(A\times B)/G$ where $A$ and $B$ are elliptic curves and $G$ is a finite abelian group acting faithfully on $A$ and $B$ such that $A/G$ is an elliptic curve and $B/G\isom \PP^1$ (see, e.g.,~\cite[Chap. VI and VIII]{Beauville-ComplexSurfaces}).  Furthermore, by the Bagnera-de Franchis classification~\cite[List VI.20]{Beauville-ComplexSurfaces}, the pair of $\gamma := |G|$ and $n:= \textup{exponent}(G)$ must be one of the following:
        	\[
        	(\gamma,n) \in \{ (2,2), (4,2), (4,4), (8,4), (3,3), (9,3), (6,6) \}\,.
        	\]
        	In all cases $n$ is the order of the canonical sheaf in $\Pic X$.

        	For $A,B,G$ as above, the universal property of the Albanese variety and~\cite[Ex IX.7(1)]{Beauville-ComplexSurfaces} imply that  the natural $k$-morphism $\Psi\colon X \to \Alb^1_X$ geometrically agrees with the projection map $(A\times B)/G \to A/G$.

            \begin{Lemma}\label{lem:BiellipticDivisors}
                Let $A$ and $B$ be elliptic curves over $K$ and let $G$ be a finite group acting faithfully on $A$ and $B$ such that $A/G$ is an elliptic curve and $B/G\isom \PP^1$.  Set $X := (A\times B)/G$ and let $\pi$ denote the projection map $X \to B/G$. If $\calL\in \Pic X$ is such that $\calL\sim_{\textup{alg}} \pi^*(\OO_{\PP^1}(m))$ for some positive integer $m$ and $\HH^0(X, \calL) \neq 0$, then $\calL \simeq \pi^*(\OO_{\PP^1}(m))$.
            \end{Lemma}
            \begin{proof}
                Since $\calL\sim_{\textup{alg}} \pi^*(\OO_{\PP^1}(m))$ and $\Pic^0_X\isom A/G\isom\Pic^0_{A/G}$, there exists a line bundle $\calL'\in \Pic^0(A/G)$ such that $\calL \isom \pi^*(\OO_{\PP^1}(m))\otimes\pi_2^*(\calL')$, where $\pi_2$ is the projection $X\to A/G$.  By assumption $\HH^0(X, \calL) \neq 0$, so $\HH^0(\PP^1, \pi_*\calL)\neq 0$, which, by the projection formula, implies that $\HH^0(\PP^1,\pi_*\pi_2^*\calL')$, and hence $\HH^0(A/G,\calL')$, are nonzero.  We complete the proof by observing that $\HH^0(A/G,\calL') \neq 0$ if and only if $\calL' \isom \OO_{A/G}$.
            \end{proof}
    
        	\begin{Proposition}\label{prop:fibration}
        		Let $X$ be a bielliptic surface over $K$.  Then there exists a genus $0$ curve $C$ and a $K$-morphism $\Phi\colon X\to C$ that is geometrically isomorphic to the projection map  $(A\times B)/G \to B/G$.
        	\end{Proposition}
            \begin{proof}
                By the geometric classification, there exist smooth elliptic curves $A$ and $B$ over $\Kbar$ such that $\Xbar\isom (A\times B)/G$.  By abuse of notation, we will also use $A$ and $B$ to refer to the algebraic equivalence class of a smooth fiber of the projection maps $\Xbar \to B/G$ and $\Xbar \to A/G$, respectively.
                
                Since $\Xbar$ has an ample divisor, the sum of the Galois conjugates of this divisor is a $K$-rational ample divisor $D$. Then by~\cite[Lemma 1.3 and Table 2]{Serrano1990}, $[D] \equiv \alpha A + \beta B \in \Num \Xbar$ for some positive $\alpha, \beta \in \frac1n\Z$.  Since the natural map $X\to{\Alb^1_X}$ is $K$-rational, taking the fiber above a closed point yields a $K$-rational divisor $F$ representing $mB\in \Num \Xbar$ for some $m>0$.   By taking a suitable integral combination of $D$ and $F$, we obtain a $K$-rational divisor $D'$ that is algebraically equivalent to $m'A$ for some positive integer $m'$.  Hence, $m'A\in (\NS \Xbar)^{\Gal(\Kbar/K)}$.  
        
		{Applying Lemma~\ref{lem:BiellipticDivisors} to the Galois conjugates of $\OO_\Xbar(m'A)$ we see that  $\OO_\Xbar(m'A)\in (\Pic \Xbar)^{\Gamma_K}$.}
		{By the Hochschild-Serre spectral sequence the cokernel of $\Pic X \to (\Pic \Xbar)^{\Gamma_K}$ injects into $\Br K$, which is torsion. Therefore some multiple of $A$ is represented by a $K$-rational divisor $D_0$.  }
                
                Since $A$ is the class of the pull back of $\OO_{B/G}(1)$ under the projection map $(A\times B)/G\to B/G)$, it follows that $\phi_{|D_0|}$ is geometrically isomorphic to the projection map $(A\times B)/G \to B/G$ composed with an $r$-uple embedding.  Thus, the image of $\phi_{|D_0|}$ is a $K$-rational genus $0$ curve $C$ and $\phi_{|D_0|}$ is the desired map. 
        	\end{proof}
	
		%%%%%%%%%%%%%%%%%%%%%%%%%%%%%%%%%%%%%%%%%%%%%%%%%%%%%%%%%%%%%%%%%%%%%%%%
        \subsubsection{Proof of Theorem~\ref{thm:MainBielliptic}}%%%%%%%%%%%%%%%
        \label{sec:ProofBielliptic}%%%%%%%%%%%%%%%%%%%%%%%%%%%%%%%%%%%%%%%%%%%%%
		%%%%%%%%%%%%%%%%%%%%%%%%%%%%%%%%%%%%%%%%%%%%%%%%%%%%%%%%%%%%%%%%%%%%%%%%
	
        	\begin{Lemma}\label{lem:bielliptic}
        		Let $X$ be a bielliptic surface of period $d$ over a number field $k$. Then $X$ satisfies $\BMprimary{nd}$ where $n$ denotes the order of the canonical sheaf in $\Pic X$.
        	\end{Lemma}
        	\begin{proof}
            Since the canonical sheaf is $n$-torsion and is defined over $k$, there exists a $\mu_n$-torsor $\pi\colon V \to X$ that is defined over $k$; by~\cite[Proposition 1]{BS-bielliptic} $V$ is a torsor under an abelian surface. Since $\pi$ has degree $n$ and $X$ has period $d$, the period of $V$ must satisfy $\Supp(\per(V))\subset \Supp(nd)$. Indeed, for any $i \in \Z$, the degree $n$ map $V \to X$ induces a map $\Alb^i_X \to \Alb^{ni}_V$. Since $\Alb^d_X(k) \ne \emptyset$ we must have $\Alb_V^{nd}(k) \ne \emptyset$. Then {Theorem~\ref{thm:PreciseMainPHS} and Case $(2)$ of Theorem~\ref{thm:QuotientOfPHS} (see Remark~\ref{rmk:PHSSatisfyAssumption})} show that $X$ satisfies $\BMprimary{nd}$. 
        	\end{proof}

		    {We now state and prove our main result on bielliptic surfaces. Theorem~\ref{thm:MainBielliptic} follows immediately.}
	
        	\begin{Theorem}\label{thm:BiellipticMainTheorem}
        		Let {$(X,\calL)$ be a nice polarized} bielliptic surface over a number field $k$.
			  If the canonical sheaf of $X$ has order $3$ or $6$, assume that $\Alb^1_X$ is not a nontrivial divisible element of $\Sha(k,\Alb^0_X)$. Then $X$ satisfies $\BMprimary{\deg(\calL)}$  
        	\end{Theorem}		

            \begin{Remark}
                If the canonical sheaf of $X$ has order $3$ or $6$, then it follows from the Bagnera-de Franchis classification that the elliptic curve $\Alb^0_X$ has $j$-invariant $0$.
            \end{Remark}

        	\begin{proof}
        		Let $n$ denote the order of the canonical sheaf of $X$. We will show that one of the four possibilities always occurs:
        		\begin{enumerate}
				\item $X$ is not locally solvable;
				\item There is a $k$-rational point on $X$;
        		\item There is a Brauer-Manin obstruction to rational points on $\Alb^1_X$; or
        		\item Every globally generated ample line bundle $\calL$ on $X$ has $\Supp(n)\subset\Supp(\deg(\calL))$.
        			              
        		\end{enumerate}

                \item In the first two cases $\BMprimary{d}$ holds trivially for every $d$. In the third case, the theorem follows from Corollary~\ref{cor:albob}. In the fourth case we have 
                $\Supp(\deg(\calL)) = \Supp(n\deg(\calL))$ and, by Corollary~\ref{cor:perioddividesdegree}, $\Supp(n\deg(\calL))\supset \Supp(n\per(X))$,
                %$\Supp(n\deg(\calL)) = \Supp(\deg(\calL))$, 
                so we can apply Lemma~\ref{lem:bielliptic} to conclude that $\BMd{\deg(\calL)}$ holds.

        		By {the adjunction formula}, the N\'eron-Severi lattice is even, so when $n \in \{ 2, 4\}$ we are in case (4) above. By the Bagnera-de Franchis classification we may therefore assume $(\gamma,n)$ is one of $(6,6)$, $(3,3)$ or $(9,3)$. Furthermore we can assume $X$ is locally soluble. Then $\Alb^1_X$ represents an element in $\Sha(k,\Alb^0_X)$. If this class is nontrivial, then Manin's theorem (cf. the proof of Proposition~\ref{prop:Bcyr}) shows that $\Alb^1_X(\A_k)^{\Br} = \emptyset$ and we are in case (3) above.
		
        		Since $X$ is locally solvable Proposition~\ref{prop:fibration} gives a $k$-morphism $\Phi\colon X \to \PP^1$. Since $\PP^1$ and $\Alb^1_X$ have $k$-points, we obtain $k$-rational fibers $A, B \in \Div(X)$ above $k$-points of $\PP^1$ and $\Alb^1_X$, respectively. All fibers of $\Psi$ are smooth genus one curves geometrically isomorphic to $B$, while the general fiber of $\Phi$ is geometrically isomorphic to $A$, but there are $3$ multiple fibers with (at least) one having multiplicity $n$~\cite[Table 2]{Serrano1990}. Let $A_0 \in \Div(\Xbar)$ denote the reduced component of a multiple fiber with multiplicity $n$. By \cite[Theorem 1.4]{Serrano1990}, the classes of $A_0$ and $\frac{n}{\gamma}B$ give a $\Z$-basis for $\Num(\Xbar)$ and the intersection pairing is given by $B^2 = A^2 = 0$ and $A \cdot B = \gamma$. The ample subset of $\Pic \Xbar$ maps to the set of positive integral linear combinations of $A_0$ and $\frac{n}{\gamma}B$ in $\Num(\Xbar)$.

        		Let $e_A$ and $e_B$ be the smallest positive integers such that the classes of $e_AA_0$ and $e_B\left(\frac{\gamma}{n}B\right)$ in $\NS(\Xbar)$ are represented by $k$-rational divisors on $X$. Since $nA_0 = A$ in $\Pic \Xbar$ we have $e_A \mid n$. Moreover, $e_A = 1$ when $(\gamma,n) = (6,6)$, because in this case $\Phi$ has a unique fiber of multiplicity $6$, which must lie over a $k$-rational point. Similarly, $\frac{\gamma}{n}\left(\frac{n}{\gamma}B\right) = B$, so $e_B \mid \frac{\gamma}{n}$. Therefore both $e_A$ and $e_B$ must divide $3$. It follows that when $e_Ae_B > 1$ every $k$-rational ample divisor on $X$ has degree divisible by $6$ and we are in case (2) above. Therefore we are reduced to considering the case $e_A = e_B = 1$, in which case we will complete the proof by showing that $X(k) \ne \emptyset$.

        		First we claim that when $e_A = 1$ the class of $A_0$ in $\NS(\Xbar)$ is represented by an effective $k$-rational divisor. In the case $n = 6$, we have seen that $A_0 \in \Div(X)$. In the case $n = 3$, let $F_1 = A_0,F_2,F_3$ denote the reduced components of the multiple fibers of $\Phi$. The class of the canonical sheaf on $X$ is represented by $2F_1 + 2F_2 + 2F_3 -2A$~\cite[Thm. 4.1]{Serrano-IsotrivialFiberedSurfaces}. Since the canonical sheaf is not trivial, the $F_i$ cannot all be linearly equivalent. The assumption that $e_A = 1$ implies that $F_1 = A_0$ is linearly equivalent to each of its Galois conjugates. The Galois action must permute the $F_i$, but it cannot act transitively, so some $F_i$ must be fixed by Galois. This $F_i$ is an effective $k$-rational divisor representing the class of $A_0$. Relabeling if necessary, we may therefore assume that $A_0$ is an effective $k$-rational divisor.

        		When $\gamma = n$, $\frac{n}{\gamma}B = B$ is a effective $k$-rational divisor intersecting $A_0$ transversally and $A_0 \cdot B = 1$, so $A_0 \cap B$ consists of a $k$-rational point. When $\gamma \ne n$ we have $(\gamma,n) = (9,3)$. In this case, $D  = A_0 + \frac{5}{3}B$ is very ample \cite[Theorem 2.2]{Serrano1990}. By Bertini's theorem \cite[Lemma V.1.2]{Hartshorne} the complete linear system $|D|$ contains a curve $C \in \Div(X)$ intersecting $A_0$ transversally. This gives a $k$-rational $0$-cycle of degree $5$ on $A_0$. On the other hand, the restriction of $\Psi$ to $A_0$ gives a degree $3$ \'etale map $A_0 \to \Alb^1_X$. As $\Alb^1_X(k) \ne \emptyset$, this gives a $k$-rational $0$-cycle of degree $3$ on $A_0$. Then $A_0$ is a genus one curve with a $k$-rational $0$-cycle of degree $1$, so it (and hence $X$) must have a $k$-point.
        	\end{proof}
		
    %%%%%%%%%%%%%%%%%%%%%%%%%%%%%%%%%%%%%%%%%%%%%%%%%%%%%%%%%%%%%%%%%%%%%%%%%%%%
    \subsection{Kummer varieties}\label{sec:TwistedKummers}%%%%%%%%%%%%%
    %%%%%%%%%%%%%%%%%%%%%%%%%%%%%%%%%%%%%%%%%%%%%%%%%%%%%%%%%%%%%%%%%%%%%%%%%%%%

    	Let $A$ be an abelian variety over $K$. Any $K$-torsor $T$ under $A[2]$ gives rise to a \defi{$2$-covering} $\rho \colon V \to A$, where $V$ is the quotient of $A\times_KT$ by the diagonal action of $A[2]$ and $\rho$ is the projection onto the first factor. Then $T = \rho^{-1}(0_A)$ and $V$ has the structure of a $K$-torsor under $A$. The class of $T$ maps to the class of $V$ under the map $\HH^1(K,A[2]) \to \HH^1(K,A)$ induced by the inclusion of group schemes $A[2] \hookrightarrow A$ and, in particular, the period of $V$ divides $2$.
	
    	Let $\sigma \colon \widetilde{V} \to V$ be the blow up of $V$ at $T \subset V$. The involution $[-1]\colon A \to A$ fixes $A[2]$ and induces involutions $\iota$ on $V$ and $\tilde{\iota}$ on $\widetilde{V}$ whose fixed point sets are $T$ and the exceptional divisor, respectively. The quotient $\widetilde{V}/\tilde{\iota}$ is geometrically isomorphic to the Kummer variety of $A$, so in particular is smooth.  We call the quotient $\widetilde{V}/\tilde{\iota}$ the \defi{Kummer variety} associated to the $2$-covering $\rho\colon V \to A$.
        \begin{Theorem}
            Let $X$ be a Kummer variety over a number field $k$. Then $X$ satisfies $\BMprimary{2}$.
        \end{Theorem}
        
        \begin{proof}
            By definition every Kummer variety admits a double cover by a smooth projective variety birational to a torsor of period dividing $2$ under an abelian variety. {Moreover $(\Br X / \Br_0 X)[2^\infty]$ is finite by~\cite[Corollary 2.8]{SZ16}}.  So the theorem follows from {Theorem~\ref{thm:PreciseMainPHS} and case (1) of Theorem~\ref{thm:QuotientOfPHS} (see  Remark~\ref{rmk:PHSSatisfyAssumption})}.
        \end{proof}

%%%%%%%%%%%%%%%%%%%%%%%%%%%%%%%%%%%%%%%%%%%%%%%%%%%%%%%%%%%%%%%%%%%%%%%%%%%%%%%%
\section{Severi-Brauer {bundles}}\label{sec:ConicBundle}%%%%%%%%%%%%%%%%%%%%%%%%%%
%%%%%%%%%%%%%%%%%%%%%%%%%%%%%%%%%%%%%%%%%%%%%%%%%%%%%%%%%%%%%%%%%%%%%%%%%%%%%%%%

{A \defi{Severi-Brauer bundle} is a nice variety $X$ together with a dominant morphism $\pi \colon X \to Y$ to a nice variety $Y$ such that the generic fiber is a smooth Severi-Brauer variety.}  

\begin{Theorem}\label{thm:ConicBundle}
    Let $\pi \colon X \to C$ be a Severi-Brauer bundle, with $C$ a curve.  Assume either
    \begin{enumerate}
        \item $g(C) = 0$ and $\pi$ is minimal of relative dimension $1$, or\label{part:genus0}
        \item $g(C) = 1$, $\pi$ is a smooth morphism, and $\Sha(k, \Alb_C^0)$ is finite.\label{part:genus1}
    \end{enumerate}
    Then degrees capture the Brauer-Manin obstruction on $X$.
\end{Theorem}
\begin{Remark}
    If $\pi\colon X\to C$ is a {minimal conic bundle} over a positive genus curve $C$ with singular fibers, then degrees do not necessarily capture the Brauer-Manin obstruction.  We construct a counterexample in~\S\ref{sec:Counterexample}.
\end{Remark}
In the case that $g(C) = 0$, it is well-known that $\Br X/\Br_0 X$ is $2$-torsion, so $X$ trivially satisfies $\BMprimary{2}$.  Then the result follows from the following lemma.

 \begin{Lemma}\label{lem:DegreesConicBundles}
     Let $\pi \colon X \to C$ be a {minimal conic bundle over a curve $C$} without a section.  Then the degree of any globally generated ample line bundle in $\Pic X$ is divisible by $4i$, where $i$ is the odd part of the index of $C$.
 \end{Lemma}
 \begin{proof}
     Since $\pi$ has no section, the class of the generic fiber in $\Br \kk(C)$ has order $2$.  Then, since $\pi\colon X \to C$ is minimal, any line bundle in $\calL \in \Pic X$ is algebraically equivalent (over $\kbar$) to $\OO_X(2aS + b\ind(C)F)$, where $a,b,\in\Z$, $S$ is a geometric section of $\pi$ (which exists by Tsen's theorem) and $F$ is the class of a fiber over a $\Kbar$-point of $C$.  Thus, we have
     \[
         \deg(\calL) = \deg(\calL_{\Kbar}) = 4a^2S^2 + 4ab\ind(C).
     \]
     Further, since there is a double cover $C'\to C$ such that the conic bundle $\pi'\colon X\times_C C' \to C'$ has a section, $S^2$ is equal to the degree of the wedge power of a rank $2$ vector bundle on $X\times_C C'$\cite[p. 30 and Prop. III.18]{Beauville-ComplexSurfaces}.  Since any degree $n$ point on $C$ yields a degree $n$ or $2n$ point on $C'$, the index of $C$ and $C'$ can differ only by a factor of $2$.  Therefore, $S^2 \in \ind(C')\Z \subset i\Z$ and so $\deg(\calL) = 4a^2S^2 + 4ab\ind(Y) \equiv 0 \pmod{4i}.$
 \end{proof}

{To prove Part~\eqref{part:genus1} of Theorem~\ref{thm:ConicBundle} we require the following result about Severi-Brauer bundles with no singular fibers.}

     \begin{Proposition}\label{prop:BSfibration}
         Let $\pi\colon X \to Y$ be a Severi-Brauer bundle 
         %with fibers of dimension $d - 1$ 
         over a nice $k$-variety $Y$ with $\pi$ a smooth morphism. Let $\calA \in \Br Y$ denote the class of the generic fiber of $\pi$ and let $d$ be the order of $\calA$. Suppose that
         \begin{enumerate}
             \item $\Sha(k,\Alb^0_Y)[d^\infty]$ is finite,
             \item $Y(k) \ne \emptyset$,
             \item The canonical map $Y(k) \to \Alb^1_Y(k)/d\Alb^0_Y(k)$ is surjective, and
             \item For every prime $v$, the evaluation map $\calA_v \colon \CH_0^{(0)}(Y_v) \to \Br k_v$ factors through $\Alb^0_Y(k_v)$.
         \end{enumerate}
         Then $X(\A)^{\Br_d}\ne \emptyset \Rightarrow X(k) \ne \emptyset$. In particular, $X$ satisfies $\BMprimary{d}$.
     \end{Proposition}
         \begin{Remark}\label{rmk:HypOnAlb}
         Hypotheses (3) and (4) of the proposition are satisfied if
         \begin{itemize}
             \item $Y$ is an abelian variety;
             \item $Y$ is a curve such that $\#Y(k) = \#\Alb_Y^0(k)$; or
             \item $Y$ is a curve such that $Y(k) \ne \emptyset$ and $\Alb_Y^0(k)$ is finite of order prime to $d$.
         \end{itemize}
         This is a slight generalization of a result in \cite[Proposition 6.5]{CPS} which proves that the Brauer-Manin obstruction is the only one on $X$ under the assumption that $Y$ is an elliptic curve with $\Sha(k,Y)$ finite.
     \end{Remark}

     \begin{proof}[Proof of Theorem~\ref{thm:ConicBundle}\eqref{part:genus1}]
        If $C(k) = \emptyset$, then the assumption on $\Sha(k, \Alb_C^0)$ implies by Manin's theorem that $C(\A_k)^{\Br} = \emptyset$.  Since $C \isom \Alb^1_X$, we may apply Corollary~\ref{cor:albob} to conclude that degrees capture the Brauer-Manin obstruction on $X$.
 
        Assume that $C(k) \neq \emptyset$ and let $d$ be the order of the generic fiber of $\pi$ in $\Br C$.  Then Proposition~\ref{prop:BSfibration} and Remark~\ref{rmk:HypOnAlb} together show that $X$ satisfies $\BMprimary{d}$ so it suffices to show that $d$ divides the degree of any globally generated ample line bundle.
         
        {Restriction to the generic fiber yields an exact sequence
        \[
            0 \to \Z F \to \NS X \to \NS X_{\eta} \to 0,
        \]
        where $F$ denotes the class of a fiber of $\pi$.  (Here we are using the assumption that $\pi$ is smooth, so there are no reducible fibers.) Since $X_{\eta}$ is a Severi-Brauer variety whose class in the Brauer group has order $d$, $\NS X_{\eta}$ is free of rank $1$ and generated by $dH$, where $H$ is a generator of $\NS \overline{X_{\eta}} \isom \NS \PP^{r} \isom \Z$. Thus any divisor class is $\NS X$ can be represented as $adH' + bF$, where $H'$ is a Zariski closure of $H$.  Since
        \begin{align*}
            (adH'+ bF)^{\dim X} = 
            &(adH')^{\dim X} + (\dim X)(ad)^{\dim X - 1}b(H')^{\dim X-1}\cdot F\\ = 
            &(ad)^{\dim X - 1}(ad\cdot H'^{\dim X} + (\dim X)b(H'^{\dim X - 1}\cdot F)),
        \end{align*}
        and $\dim X > 1$, this shows that the degree of any line bundle is divisible by $d$.}
     \end{proof}

     \begin{proof}[Proof of Proposition~\ref{prop:BSfibration}]
         To ease notation set $A = \Alb^0_Y$. Choose $P_0 \in Y(k)$ and use this to define a $k$-morphism $\iota\colon Y \to A$ sending $P$ to the class of the $0$-cycle $P-P_0$. Suppose that $X(\A_k)^{\Br_d}\ne \emptyset$. Fix a point $(Q_v) \in X(\A_k)^{\Br_d}$ and let $(P_v) := (\pi(Q_v))$. Then $(\iota(P_v)) \in A(\A_k)^{\Br_d}$ by functoriality.  Since $\iota(P_v)$ is orthogonal to $(\Br A)[d^\infty]$ it follows from descent theory (e.g., \cite[Thm. 6.1.2]{TorsorsAndRationalPoints}) that, for every $n$, $\iota(P_v)$ lifts to an adelic point on some $d^n$-covering of $A$. (For the definition of an $N$-covering see \cite[Def. 3.3.1]{TorsorsAndRationalPoints}.) Because $\Sha(k,A)[d^\infty]$ is finite, there is some $n$ such that every $d^n$-covering of $A$ which lifts to a locally soluble $d^{n+1}$-covering of $A$ is of the form $x \mapsto d^{n}x + Q$ for some $Q \in A(k)$. In particular, there is some $Q \in A(k)$ such that $(\iota(P_v) - Q) \in dA(\A_k)$. By assumption (3) there is some $P \in Y(k)$ such that $(\iota(P_v) - \iota(P)) \in dA(\A_k)$. Now by assumption (4) and the fact that $d\calA = 0$ it follows that, for every prime $v$, $\calA(P)\otimes k_v = \calA(P_v)$. Note that $\calA(P_v) = 0\in \Br k_v$ since the fiber $X_{P_v}$ has a $k_v$-point. Thus, $\calA(P)\otimes k_v = 0$ for every $v$. We conclude that the Severi-Brauer variety $X_P$ must be everywhere locally soluble and thus, by the Albert-Brauer-Hasse-Noether theorem, that $X_P(k) \neq \emptyset$.
     \end{proof}
    %%%%%%%%%%%%%%%%%%%%%%%%%%%%%%%%%%%%%%%%%%%%%%%%%%%%%%%%%%%%%%%%%%%%%%%%%%%%
	\subsection{A counterexample}\label{sec:Counterexample}
    %%%%%%%%%%%%%%%%%%%%%%%%%%%%%%%%%%%%%%%%%%%%%%%%%%%%%%%%%%%%%%%%%%%%%%%%%%%%
	
        \begin{Theorem}\label{thm:ConicBundleCounterexample}
            There exists a conic bundle surface $\pi\colon X \to E$ over an elliptic curve, such that X has an ample globally generated line bundle of degree 12 and $X$ does not satisfy $\BMprimary{6}$. In particular, degrees do not capture the Brauer-Manin obstruction.
        \end{Theorem}

        \begin{proof}
            Let $E/\Q$ be an elliptic curve with a $\Q$-rational cyclic subgroup $Z$ of order $5$ such that $\Gal(\Q(Z)/\Q) = \Z/4$, {with $E(\Q) = \{O\}$} and with $\Sha(\Q, E)$ finite, e.g., the curve with LMFDB label 11.a1\cite[\href{http://www.lmfdb.org/EllipticCurve/Q/11.a2}{Elliptic Curve 11.a2}]{LMFDB}.  Then $Z = O\cup P$, where $P$ is a degree $4$ closed point.   Let $p_1$ and $p_2$ be two primes that are congruent to $1$ modulo $8$ and that split completely in $\kk(P)$.  Finally choose a prime $q \neq p_i$ that is $1$ modulo $8$ and that is a nonsquare modulo $p_1$ and modulo $p_2$.

            We will use this chosen data to construct a conic bundle $X$ over $E$ with the required properties.  Let $\calE$ be the rank $3$ vector bundle $\OO_E \oplus \OO_E \oplus \OO_E(2O)$.  Since $4 O\sim P$, there exists a section $s_2\in \HH^0(E, \OO_E(2O)^{\otimes2})$ such that $\divv(s_2) = P$ and such that $s_2(O) = -p_1p_2$.  Let $s_0 := 1$ and $s_1 := -q$ thought of as sections of $\HH^0(E, \OO_E^{\otimes 2})$.  Then we define $\pi \colon X\to E$ to be the conic bundle given by the vanishing of the section $s:=s_0 + s_1 + s_2\in \HH^0(E, \Sym^2\calE)$ inside $\PP(\calE)$; this conic bundle is smooth by~\cite[Lemma 3.1]{Poonen-Chatelet}. 

	        Since $s$ is invertible on $E - P$, the morphism $\pi$ is smooth away from $P$. Because $P$ is a single closed point, it follows that $\pi^*\colon \Br E \to \Br X$ is surjective (see \cite[Proposition 2.2(i)]{CPS}). Hence, $X(\A_{\Q})^{\Br_{6}} \neq\emptyset$ if and only if $\pi(X(\A_{\Q}))\cap E(\A_{\Q})^{\Br_{6}} \neq \emptyset$. 

	        Since the $p_i$ split completely in $\kk(P)$, the connected components of $P \otimes k_{p_i}$ are points in $E(\Q_{p_i})$. Consider an adelic point $(P_v)\in E(\A_{\Q})$ with $P_v = O$ for all $v\neq p_1,p_2$ and $P_v$ a connected component of $P \otimes k_v$ for $v = p_1,p_2$. Then $(P_v)$ is $5$-torsion and, since $5$ is relatively prime to $6$, the point $(P_v)$ is $6$-divisible in $E(\A_\Q)$. Since $\Sha(\Q,E)$ is finite, descent theory implies that $(P_v) \in E(\A_\Q)^{\Br_{6}}$  (cf. the proof of Proposition~\ref{prop:BSfibration}).

            Since $p_i$ splits completely in $\kk(P)$, each connected component of $X_P \otimes \Q_{p_i}$ is a pair of lines intersecting at a unique point, which must be defined over $\Q_{p_i}$. In particular, $P_{p_i}\in \pi(X(\Q_{p_i}))$ for $i = 1, 2$.  Further, by construction, the fiber above $O$ is given by
            \[
                X_O : w_0^2 - qw_1^2 = p_1p_2 w_2^2.
            \]
            By the choice of $p_1, p_2$ and $q$, $X_O(\Q_{v})\neq \emptyset$ exactly when $v\neq p_1,p_2$.  Hence, $(P_v)\in \pi(X(\A_\Q))$ and so $X(\A_{\Q})^{\Br_{6}} \neq\emptyset$.

            {To prove that $X$ does not satisfy $\BMprimary{6}$,} it remains to show that $X(\A_{\Q})^{\Br} = \emptyset$, which, by the same argument as above, is equivalent to showing that $\pi(X(\A_{\Q})) \cap E(\A_{\Q})^{\Br} = \emptyset$.  Since $\Sha(k, E)$ is finite, $E(\A_{\Q})^{\Br} = \prod_{p<\infty}\{O\} \times C_O,$ where $C_O \subset E(\R)$ is the connected component of $O$.  However, {since $s_2(P) = 0$ and $q\notin \F_{p_i}^{\times2}$ for $i=1,2$}, $X_O(\Q_{p_1}) = X_O(\Q_{p_2}) = \emptyset,$ so $X(\A_{\Q})^{\Br} = \emptyset$.

            Now we will show that $X$ has a globally generated $k$-rational ample line bundle of degree $12$.  Consider the morphism $f\colon X\to \PP^3 \times \PP^1$ that is the composition of the following maps 
            \[
                X \hookrightarrow \PP(\OO_E \oplus \OO_E \oplus \OO_E(2\infty)) \to \PP(\OO_E^{\oplus4}) \isom \PP^3 \times E \stackrel{(\textup{id}, x)}{\longrightarrow} \PP^3 \times \PP^1,
            \]
            where the second map (from left to right) is induced by a surjection $\OO_E^{\oplus2} \to \OO_E(2O)$ and the morphism $g$ that is the composition of $f$ with the Segre embedding $\PP^3 \times \PP^1\hookrightarrow \PP^7$. Note that $f$, and therefore $g$, is $2$-to-$1$ onto its image, so $\calL := g^*\OO_{\PP^7}(1)$ is a globally generated ample line bundle on $X$.

            The image of $f$ is the complete intersection of a $(2,0)$ hypersurface, which comes from the quadric relation $s$, and a $(1,1)$ hypersurface, which comes from the kernel of $\OO_E^{\oplus2} \to \OO_E(2O)$.  Since a quadric surface in $\PP^3$ is geometrically isomorphic to $\PP^1 \times \PP^1$, the image of $g$ is geometrically isomorphic to a hyperplane intersected with the image $Y$ of $\PP^1\times \PP^1 \times \PP^1 \hookrightarrow\PP^7$.  Since the Hilbert polynomial of $Y$ is $(1 + x)^3$, the degree of $Y$, and hence the degree of $\im g$, is $3! = 6$.  Therefore, $\deg(\calL) = 12$ and so degrees do \emph{not} capture the Brauer-Manin obstruction to rational points on $X$.
        \end{proof}

%%%%%%%%%%%%%%%%%%%%%%%%%%%%%%%%%%%%%%%%%%%%%%%%%%%%%%%%%%%%%%%%%%%%%%%%%%%%%%%%
%%%%%%%%%%%%%%%%%%%%%%%%%%%%%%%%%%%%%%%%%%%%%%%%%%%%%%%%%%%%%%%%%%%%%%%%%%%%%%%%
\appendix
\section{Proof of Theorem~\ref{thm:KummerGeneral}, By Alexei N. Skorobogatov}
%%%%%%%%%%%%%%%%%%%%%%%%%%%%%%%%%%%%%%%%%%%%%%%%%%%%%%%%%%%%%%%%%%%%%%%%%%%%%%%%
%%%%%%%%%%%%%%%%%%%%%%%%%%%%%%%%%%%%%%%%%%%%%%%%%%%%%%%%%%%%%%%%%%%%%%%%%%%%%%%%

The following statement equivalent to Theorem~\ref{thm:KummerGeneral}
is a common generalisation of 
\cite[Thm. 3.3]{SZ16} of Skorobogatov and Zarhin and Theorem~\ref{thm:MainKummer} of Creutz and Viray.

Let $\Br(X)_{\rm odd}$ be
the subgroup of $\Br(X)$ formed by the elements of odd order.

\begin{Theorem}
Let $A$ be an abelian variety of dimension $>1$ over a number field $k$.
Let $X$ be the Kummer variety attached to a 2-covering of $A$.
If $B$ is a subgroup of $\Br(X)$ such that $X(\A_k)^B\not=\emptyset$,
then $X(\A_k)^{B+\Br(X)_{\rm odd}}\not=\emptyset$.
\end{Theorem}

\begin{proof}
By \cite[Corollary 2.8]{SZ16} the group $\Br(X)/\Br_0(X)$ is finite.
Hence we can assume without loss of generality
that $B$ is finite. Replacing $B$ by its 2-primary torsion subgroup we
can assume that $B\subset\Br(X)[2^\infty]$. There exists
an odd integer $n$ such that $\Br(X)[n]$ and $\Br(X)_{\rm odd}$
have the same image in $\Br(X)$.

By the definition of a Kummer variety \cite[Def. 2.1]{SZ16} there exists a 
double cover $\pi:Y'\to X$, where $\sigma:Y'\to Y$ is the blowing-up of 
a 2-covering $f:Y\to A$ of $A$ in 
$f^{-1}(0)$, such that the branch locus of $\pi$ is the exceptional divisor of $\sigma$.
Since $n$ is odd and $Y$ is a torsor for $A$ of period dividing 2, 
we have a well defined morphism $[n]:Y\to Y$
compatible with multiplication by $n$ on $A$.

Let $(P_v)\in X(\A_k)$.
For each $v$ there is a class $\alpha_v\in\HH^1(k_v,\mu_2)=k_v^*/k_v^{*2}$
such that $P_v$ lifts to a $k_v$-point on the quadratic twist $Y'_{\alpha_v}$,
which is a variety defined over $k_v$. By weak approximation we can assume
that $\alpha_v$ comes from $k^*/k^{*2}$ and hence $Y'_{\alpha_v}$
is actually defined over $k$.
Let $\pi_{\alpha_v}:Y'_{\alpha_v}\to X$ be the natural double cover.

Let $D=\pi_{\alpha_v}^*(B)$ be the image of $B$ in $\Br(Y_{\alpha_v})$. 
We need the following corollary of Lemma~\ref{lem:killntors}.

\begin{Lemma}\label{lem:appendix}
There exists a positive integer $s$ such that 
$[n^s]^*:\Br(Y_{\alpha_v})\to\Br(Y_{\alpha_v})$ restricted to
$D$ is the identity map on $D$.
\end{Lemma}

\begin{proof}
By Lemma~\ref{lem:killntors} for $d=2$ there is 
a positive integer $a$ such that $[n^a]$
induces the identity on the quotient of $D$ by $D\cap\Br_0(Y_{\alpha_v})$. 
Hence for each $\calA\in D$ there is an $\calA_0\in\Br_0(Y_{\alpha_v})$ such that
$[n^a]^*\calA=\calA+\calA_0$. Let $b$ be the product of orders of $\calA_0$,
where $\calA$ ranges over all elements of $D$. Then $s=ab$ 
satisfies the conclusion of the lemma.
\end{proof}

Now assume that $(P_v)\in X(\A_k)^B$. For each place $v$ of $k$
choose $R_v\in Y_{\alpha_v}(k_v)$ such that $\pi_{\alpha_v}(R_v)=P_v$. 
In the proof of \cite[Thm. 3.3]{SZ16} instead of $M_v=[n]R_v$ put
$M_v=[n^s]R_v$, where $s$ is as in Lemma \ref{lem:appendix}.
This definition has all the properties needed for the proof of \cite[Thm. 3.3]{SZ16}
with the additional property
\begin{equation}
\calA(M_v)=([n^s]^*\calA)(R_v)=\calA(R_v) \label{eq1}
\end{equation} for each $\calA\in \pi_{\alpha_v}^*(B)$.
We now proceed exactly as in the proof of \cite[Thm. 3.3]{SZ16}.
By a small deformation we can assume that $M_v$ avoids the exceptional 
divisor of $Y'_{\alpha_v}\to Y_{\alpha_v}$. Then $M_v$ lifts to a unique
point $M'_v\in Y'_{\alpha_v}(k_v)$.
Let $Q_v=\pi_{\alpha_v}(M'_v)\in X(k_v)$. By (\ref{eq1}) 
for each $\calB\in B$ we have $\calB(Q_v)=\calB(P_v)$, hence
$(Q_v)\in X(\A_k)^B$.
The proof of \cite[Thm. 3.3]{SZ16} shows that $(Q_v)\in X(\A_k)^{\Br(X)_{\rm odd}}$.
Thus $(Q_v)\in X(\A_k)^{B+\Br(X)_{\rm odd}}$.
\end{proof}

%%%%%%%%%%%%%%%%%%%%%%%%%%%%%%%%%%%%%%%%%%%%%%%%%%%%%%%%%%%%%%%%%%%%%%%%%%%%%%%%
%%%%%%%%%%%%%%%%%%%%%%%%%%%%%%%%%%%%%%%%%%%%%%%%%%%%%%%%%%%%%%%%%%%%%%%%%%%%%%%%
%%%%%%%%%%%%%%%%%%%%%%%%%%%%%      BIBLIOGRAPHY     %%%%%%%%%%%%%%%%%%%%%%%%%%%%
%%%%%%%%%%%%%%%%%%%%%%%%%%%%%%%%%%%%%%%%%%%%%%%%%%%%%%%%%%%%%%%%%%%%%%%%%%%%%%%%
%%%%%%%%%%%%%%%%%%%%%%%%%%%%%%%%%%%%%%%%%%%%%%%%%%%%%%%%%%%%%%%%%%%%%%%%%%%%%%%%

\end{document}